\documentclass{article}

\usepackage{psfrag}
\usepackage[dvips]{graphicx}
\usepackage{subfigure}
\usepackage{amsmath}
\usepackage{amsthm}
\usepackage{amssymb}
\usepackage{enumerate}
\usepackage{setspace}
\usepackage{epsfig}
\usepackage{fancybox}
\usepackage{float}
\usepackage{times}
\usepackage{array}
\usepackage{multind}
\usepackage{verbatim}
\usepackage{color}
\usepackage{setspace}
\usepackage{kbordermatrix}
\usepackage{latexsym}
\usepackage[normalem]{ulem}
\usepackage{setspace}
\usepackage{subfigure}
\usepackage[section,subsection,subsubsection]{extraplaceins}

\newtheorem{theorem}{Theorem}

\newtheorem{lemma}{Lemma}

\theoremstyle{definition}

\theoremstyle{remark}

\newfont{\valami}{ptmr8r scaled 1200} 
\newfont{\kisvalami}{ptmr8r scaled 1000}

\newcommand{\dof}{\backslash}

\begin{document} 
\author{
Konstantinos Papalamprou \\
London School of Economics
\and
Leonidas Pitsoulis  \\
Aristotle University of Thessaloniki 
}
\title{Binary Matroids with Graphic Cocircuits}
\date{}
\maketitle
 \begin{abstract}
An excluded minor characterization for the class of binary signed-graphic matroids with graphic cocircuits is provided.  
In this report we present the necessary computations for the case analysis in the proof.
\end {abstract}

\section*{An excluded minor characterization}
The complete list of regular excluded minors for signed-graphic matroids is provided in~\cite{SliQinZh:09}, specifically: 
\begin{theorem} 
A regular matroid $M$ is signed-graphic if and only if $M$ has no minor isomorphic to $M^{*}(G_1),\ldots,M^{*}(G_{29}),R_{15} or R_{16}$. 
\end{theorem}
The matroids $M^{*}(G_1),\ldots,M^{*}(G_{29})$ are the cographic matroids of the 29 non-separable forbidden minors for projective planar graphs, while $R_{15}$ and $R_{16}$ are 
two special matroids whose binary compact representation matrices are given in~\cite{SliQinZh:09} and in the next section of this Technical Report.

Clearly one could easily produce the complete list of binary excluded minors for signed-graphic matroids by adding to the list of the above 31 regular excluded minors
the binary excluded minors for regular matroids (i.e. $F_7$ and $F_7^{*}$), since any binary signed-graphic matroid is also regular.
\begin{theorem} \label{th_rew}
A binary matroid $M$ is signed-graphic if and only if $M$ has no minor isomorphic to $M^{*}(G_1),\ldots,M^{*}(G_{29}), R_{15}$, $R_{16}$, $F_7$ or $F_7^{*}$. 
\end{theorem}

We define a cocircuit $Y$ of a matroid $M$ to be \emph{graphic} if the matroid $M\dof{Y}$ is graphic. The main result in this report
is the complete list of excluded minors for the class of binary signed-graphic matroids with graphic cocircuits. Of importance for the proof of that result, 
Theorem~\ref{th_ll3} here, is the following Lemma.
\begin{lemma}\label{lem_1}
If $N$ is a minor of the matroid $M$ then for any cocircuit $C_{N}\in\mathcal{C}(N^{*})$ there
exists a cocircuit $C_{M}\in\mathcal{C}(M^{*})$ such that $N\dof C_{N}$ is a minor of 
$M\dof C_{M}$. 
\end{lemma}
\begin{proof}
If $N=M\dof X / Y$ then by duality $N^{*}=M / X \dof Y$. Therefore by the definitions
of contraction and deletion of a set, we have
that for any cocircuit $C_{N}\in\mathcal{C}(N^{*})$ there exists a cocircuit $C_{M}\in\mathcal{C}(M^{*})$ such
that 
\begin{itemize}
\item[(i)] $C_{N} \subseteq C_{M}$,
\item[(ii)] $E(N) \cap C_{M} = C_{N}$,
\end{itemize}
which in turn imply that  $C_{M} - C_{N} \subseteq X$. So we have 
\[
M\dof C_{M}=  M \dof \{C_{M} - C_{N}\} \dof C_{N} \succeq N \dof C_{N}
\]
\end{proof}

\begin{theorem}\label{th_ll3}
Let $M$ be a binary matroid such that all its cocircuits are graphic. Then, $M$ is signed-graphic if and only if $M$ has no minor 
isomorphic to $M^{*}(G_{17})$, $M^{*}(G_{19})$, $F_7$ or $F_7^{*}$.
\end{theorem}
\begin{proof}
$M$ must contain a minor isomorphic to some matroid in the set
\[
\mathcal{M}=\{M^{*}(G_1),\ldots,M^{*}(G_{16}),M^{*}(G_{18}),
M^{*}(G_{20}),\ldots,M^{*}(G_{29}), R_{15}^{*}, R_{16}^{*}\}.
\]
By case analysis, verified also by the MACEK software~\cite{Hlileny:07},
it can be shown that for each  matroid  $N\in \mathcal{M}$ there exists
a cocircuit $Y_{N}\in \mathcal{C}(N^{*})$ such that the matroid $N\dof {Y_{N}}$ 
contains an $M^{*}(K_{3,3})$ or an $M^{*}(K_{5})$ as a minor, which implies that $N\dof Y_{N}$ is not graphic. 
Therefore, by Lemma~\ref{lem_1}, there is a cocircuit $Y_{M}\in\mathcal{C}(M^{*})$ such that
$N\dof Y_{N}$ is a minor of $M\dof Y_{M}$.  Thus, $M\dof Y_{M}$
is not graphic which is in contradiction with our assumption that $M$ has graphic cocircuits.
\end{proof}

As already mentioned, this Technical Report is mainly devoted to the computations performed using the MACEK software~\cite{Hlileny:07} 
appearing in the proof of Theorem~\ref{th_ll3}. These computations are provided in detail in the next section.

\section*{MACEK computations}

Each case, i.e. matroid in $\mathcal{M}$ in the proof of Theorem~\ref{th_ll3}, will be examined separately and specifically:
\begin{itemize}
\item \noindent for each cographic matroids in $\mathcal{M}$, a compact representation matrix of its dual graphic matroid along with the associated graph 
($(G_1,\ldots, G_{16},G_{18},G_{20},\ldots,G_{29})$) are provided.
It is clear that due to matroid duality, it is enough to find a circuit $C$ in each $M\in \{M(G_1),\ldots,M(G_{16}),M(G_{18}),M(G_{20}),\ldots,M(G_{29})\}$ such that $M/C$ contains an $M(K_{3,3})$- or an $M(K_{5})$-minor.
The advantage of working with the duals of the cographic matroids in $\mathcal{M}$ is that someone could graphically see that by contracting 
a cycle (i.e. the one corresponding to $C$) in the associated graph, a minor isomorphic to $K_{3,3}$ or $K_5$ is contained in the resulting graph. Therefore, in that case, the MACEK computations may be seen just as a validation tool.  

\item \noindent for each of the two non-cographic matroid in $\mathcal{M}$ (i.e. $R_{15}$ and $R_{16}$), a compact representation matrix is provided
along with the cocircuit to be deleted. The MACEK commandss and the outputs showing that each of the resulting matroids contains an  $M^{*}(K_{3,3})$ 
or an $M^{*}(K_{5})$ as a minor are given.
\end{itemize}

\subsubsection*{The matroid $M(G_1)$:}
\FloatBarrier
\begin{figure*}[h]
\raisebox{-10ex}{\resizebox*{0.6\width}{!}{
\psfrag{r1}{\footnotesize $r_1$}
\psfrag{r2}{\footnotesize $r_2$}
\psfrag{r3}{\footnotesize $r_3$}
\psfrag{r4}{\footnotesize $r_4$}
\psfrag{r5}{\footnotesize $r_5$}
\psfrag{r6}{\footnotesize $r_6$}
\psfrag{r7}{\footnotesize $r_7$}
\psfrag{s1}{\footnotesize $s_1$}
\psfrag{s2}{\footnotesize $s_2$}
\psfrag{s3}{\footnotesize $s_3$}
\psfrag{s4}{\footnotesize $s_4$}
\psfrag{s5}{\footnotesize $s_5$}
\psfrag{s6}{\footnotesize $s_6$}
\psfrag{s7}{\footnotesize $s_7$}
\psfrag{s8}{\footnotesize $s_{8}$}
\psfrag{s9}{\footnotesize $s_{9}$}
\psfrag{s10}{\footnotesize $s_{10}$}
\psfrag{s11}{\footnotesize $s_{11}$}
\includegraphics [scale=0.8] {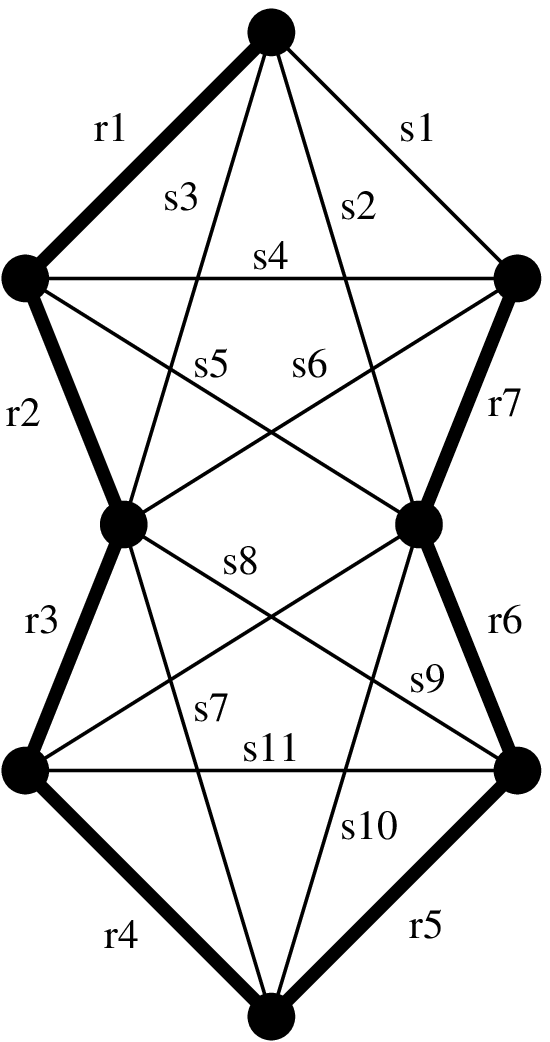}}}
$\mspace{30mu}$
$
g_1=
\kbordermatrix{\mbox{}&s_1&s_2&s_3&s_4&s_5&s_6&s_7&s_8&s_9&s_{10}&s_{11}\\
r_1 & {\;\,\! 1}  	   & {\;\,\! 1}          & {\;\,\! 1}          & {\;\,\! 0}          & {\;\,\! 0}          & {\;\,\! 0}          & {\;\,\! 0}         & {\;\,\! 0}          & {\;\,\! 0}               & {\;\,\! 0}          & {\;\,\! 0} \\
r_2 & {\;\,\! 1}           & {\;\,\! 1}          & {\;\,\! 1}          & {\;\,\! 1}          & {\;\,\! 1}          & {\;\,\! 0}          & {\;\,\! 0}         & {\;\,\! 0}          & {\;\,\! 0}               & {\;\,\! 0}          & {\;\,\! 0}  \\ 
r_3 & {\;\,\! 1}           & {\;\,\! 0}          & {\;\,\! 1}          & {\;\,\! 1}          & {\;\,\! 1}          & {\;\,\! 1}          & {\;\,\! 1}         & {\;\,\! 1}          & {\;\,\! 0}               & {\;\,\! 0}          & {\;\,\! 0} \\ 
r_4 & {\;\,\! 1}           & {\;\,\! 0}          & {\;\,\! 1}          & {\;\,\! 1}          & {\;\,\! 1}          & {\;\,\! 1}          & {\;\,\! 1}         & {\;\,\! 1}          & {\;\,\! 1}               & {\;\,\! 0}          & {\;\,\! 1}  \\ 
r_5 & {\;\,\! 1}           & {\;\,\! 0}          & {\;\,\! 1}          & {\;\,\! 1}          & {\;\,\! 1}          & {\;\,\! 1}          & {\;\,\! 0}         & {\;\,\! 1}          & {\;\,\! 1}               & {\;\,\! 1}          & {\;\,\! 1}   \\ 
r_6 & {\;\,\! 1}           & {\;\,\! 0}          & {\;\,\! 1}          & {\;\,\! 1}          & {\;\,\! 1}          & {\;\,\! 1}          & {\;\,\! 0}         & {\;\,\! 0}          & {\;\,\! 1}               & {\;\,\! 1}          & {\;\,\! 0}  \\ 
r_7 & {\;\,\! 1}           & {\;\,\! 0}          & {\;\,\! 0}          & {\;\,\! 1}          & {\;\,\! 0}          & {\;\,\! 1}          & {\;\,\! 0}         & {\;\,\! 0}          & {\;\,\! 0}               & {\;\,\! 0}          & {\;\,\! 0}   
}   
$
\label{fig_g1}
\end{figure*}
\noindent
\begin{center}$M(G_1)/\{r_1,s_1,s_3\}$ contains an $M(K_5)$-minor.\end{center}
{\tt \scriptsize {\bf Command:} ./macek -pGF2 '!contract 1;!contract -1;!contract -3;!minor' g1 '\{grK5,grK33\}'}\\
{\tt \scriptsize {\bf Output: }The \#1 matroid [g1$\sim$c1$\sim$c-1$\sim$c-3] +HAS+ minor \#1 [grK5] in the list \{grK5 grK33\}.} \\

\subsubsection*{The matroid $M(G_2)$:}
\FloatBarrier
\begin{figure*}[h]
\raisebox{-10ex}{\resizebox*{0.6\width}{!}{
\psfrag{r1}{\footnotesize $r_1$}
\psfrag{r2}{\footnotesize $r_2$}
\psfrag{r3}{\footnotesize $r_3$}
\psfrag{r4}{\footnotesize $r_4$}
\psfrag{r5}{\footnotesize $r_5$}
\psfrag{r6}{\footnotesize $r_6$}
\psfrag{r7}{\footnotesize $r_7$}
\psfrag{r8}{\footnotesize $r_8$}
\psfrag{s1}{\footnotesize $s_1$}
\psfrag{s2}{\footnotesize $s_2$}
\psfrag{s3}{\footnotesize $s_3$}
\psfrag{s4}{\footnotesize $s_4$}
\psfrag{s5}{\footnotesize $s_5$}
\psfrag{s6}{\footnotesize $s_6$}
\psfrag{s7}{\footnotesize $s_7$}
\psfrag{s8}{\footnotesize $s_8$}
\psfrag{s9}{\footnotesize $s_9$}
\includegraphics [scale=0.8] {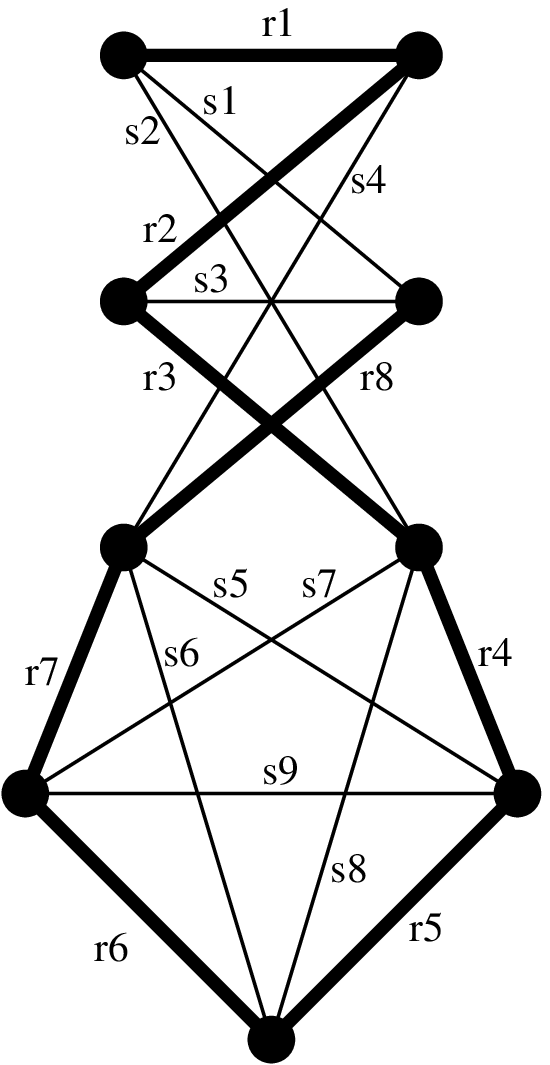}}}
$\mspace{30mu}$
$
g_2=
\kbordermatrix{\mbox{}&s_1&s_2&s_3&s_4&s_5&s_6&s_7&s_8&s_9\\
r_1 & {\;\,\! 1}  	   & {\;\,\! 1}          & {\;\,\! 0}          & {\;\,\! 0}          & {\;\,\! 0}          & {\;\,\! 0}          & {\;\,\! 0}         & {\;\,\! 0}          & {\;\,\! 0}  \\
r_2 & {\;\,\! 1}           & {\;\,\! 1}          & {\;\,\! 0}          & {\;\,\! 1}          & {\;\,\! 0}          & {\;\,\! 0}          & {\;\,\! 0}         & {\;\,\! 0}          & {\;\,\! 0}  \\ 
r_3 & {\;\,\! 1}           & {\;\,\! 1}          & {\;\,\! 1}          & {\;\,\! 1}          & {\;\,\! 0}          & {\;\,\! 0}          & {\;\,\! 0}         & {\;\,\! 0}          & {\;\,\! 0}   \\ 
r_4 & {\;\,\! 1}           & {\;\,\! 0}          & {\;\,\! 1}          & {\;\,\! 1}          & {\;\,\! 0}          & {\;\,\! 0}          & {\;\,\! 1}         & {\;\,\! 1}          & {\;\,\! 0}    \\ 
r_5 & {\;\,\! 1}           & {\;\,\! 0}          & {\;\,\! 1}          & {\;\,\! 1}          & {\;\,\! 1}          & {\;\,\! 0}          & {\;\,\! 1}         & {\;\,\! 1}          & {\;\,\! 1}    \\ 
r_6 & {\;\,\! 1}           & {\;\,\! 0}          & {\;\,\! 1}          & {\;\,\! 1}          & {\;\,\! 1}          & {\;\,\! 1}          & {\;\,\! 1}         & {\;\,\! 0}          & {\;\,\! 1}    \\ 
r_7 & {\;\,\! 1}           & {\;\,\! 0}          & {\;\,\! 1}          & {\;\,\! 1}          & {\;\,\! 1}          & {\;\,\! 1}          & {\;\,\! 0}         & {\;\,\! 0}          & {\;\,\! 0}   \\
r_8 & {\;\,\! 1}           & {\;\,\! 0}          & {\;\,\! 1}          & {\;\,\! 0}          & {\;\,\! 0}          & {\;\,\! 0}          & {\;\,\! 0}         & {\;\,\! 0}          & {\;\,\! 0}   
}   
$
\label{fig_g2}
\end{figure*}
\noindent
\begin{center} $M(G_2)/\{r_4,r_5,s_9\}$ contains an $M(K_{3,3})$-minor.\end{center}
{\tt \scriptsize {\bf Command:} ./macek -pGF2 '!contract 4;!contract 5;!contract -9;!minor' g2 '\{grK5,grK33\}'}\\
{\tt \scriptsize {\bf Output: }The \#1 matroid [g2$\sim$c4$\sim$c5$\sim$c-9] +HAS+ minor \#2 [grK33] in the list \{grK5 grK33\}.} \\

\subsubsection*{The matroid $M(G_3)$:}
\FloatBarrier
\begin{figure*}[h]
\raisebox{-10ex}{\resizebox*{0.6\width}{!}{
\psfrag{r1}{\footnotesize $r_1$}
\psfrag{r2}{\footnotesize $r_2$}
\psfrag{r3}{\footnotesize $r_3$}
\psfrag{r4}{\footnotesize $r_4$}
\psfrag{r5}{\footnotesize $r_5$}
\psfrag{r6}{\footnotesize $r_6$}
\psfrag{r7}{\footnotesize $r_7$}
\psfrag{r8}{\footnotesize $r_8$}
\psfrag{s1}{\footnotesize $s_1$}
\psfrag{s2}{\footnotesize $s_2$}
\psfrag{s3}{\footnotesize $s_3$}
\psfrag{s4}{\footnotesize $s_4$}
\psfrag{s5}{\footnotesize $s_5$}
\psfrag{s6}{\footnotesize $s_6$}
\psfrag{s7}{\footnotesize $s_7$}
\psfrag{s8}{\footnotesize $s_8$}
\psfrag{s9}{\footnotesize $s_9$}
\psfrag{s10}{\footnotesize $s_{10}$}
\includegraphics [scale=0.8] {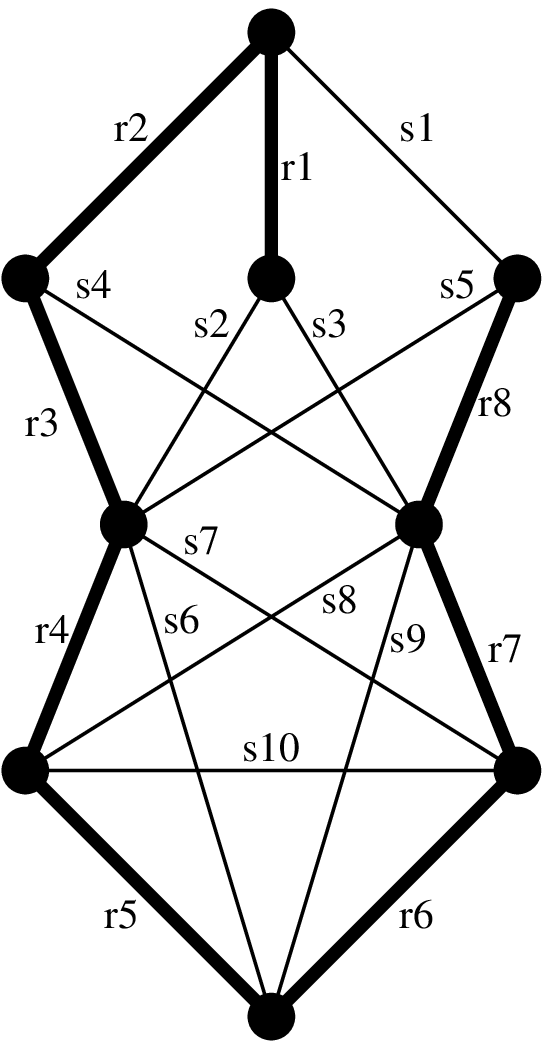}}}
$\mspace{30mu}$
$
g_3=
\kbordermatrix{\mbox{}&s_1&s_2&s_3&s_4&s_5&s_6&s_7&s_8&s_9&s_{10}\\
r_1 & {\;\,\! 0}  	   & {\;\,\! 1}          & {\;\,\! 1}          & {\;\,\! 0}          & {\;\,\! 0}          & {\;\,\! 0}          & {\;\,\! 0}         & {\;\,\! 0}          & {\;\,\! 0}    & {\;\,\! 0} \\
r_2 & {\;\,\! 1}           & {\;\,\! 1}          & {\;\,\! 1}          & {\;\,\! 0}          & {\;\,\! 0}          & {\;\,\! 0}          & {\;\,\! 0}         & {\;\,\! 0}          & {\;\,\! 0}    & {\;\,\! 0} \\ 
r_3 & {\;\,\! 1}           & {\;\,\! 1}          & {\;\,\! 1}          & {\;\,\! 1}          & {\;\,\! 0}          & {\;\,\! 0}          & {\;\,\! 0}         & {\;\,\! 0}          & {\;\,\! 0}    & {\;\,\! 0} \\ 
r_4 & {\;\,\! 1}           & {\;\,\! 0}          & {\;\,\! 1}          & {\;\,\! 1}          & {\;\,\! 1}          & {\;\,\! 1}          & {\;\,\! 1}         & {\;\,\! 0}          & {\;\,\! 0}    & {\;\,\! 0} \\ 
r_5 & {\;\,\! 1}           & {\;\,\! 0}          & {\;\,\! 1}          & {\;\,\! 1}          & {\;\,\! 1}          & {\;\,\! 1}          & {\;\,\! 1}         & {\;\,\! 1}          & {\;\,\! 0}    & {\;\,\! 1} \\ 
r_6 & {\;\,\! 1}           & {\;\,\! 0}          & {\;\,\! 1}          & {\;\,\! 1}          & {\;\,\! 1}          & {\;\,\! 0}          & {\;\,\! 1}         & {\;\,\! 1}          & {\;\,\! 1}    & {\;\,\! 1} \\ 
r_7 & {\;\,\! 1}           & {\;\,\! 0}          & {\;\,\! 1}          & {\;\,\! 1}          & {\;\,\! 1}          & {\;\,\! 0}          & {\;\,\! 0}         & {\;\,\! 1}          & {\;\,\! 1}    & {\;\,\! 0}  \\
r_8 & {\;\,\! 1}           & {\;\,\! 0}          & {\;\,\! 0}          & {\;\,\! 0}          & {\;\,\! 1}          & {\;\,\! 0}          & {\;\,\! 0}         & {\;\,\! 0}          & {\;\,\! 0}    & {\;\,\! 0}  
}   
$
\label{fig_g3}
\end{figure*}
\noindent
\begin{center} $M(G_3)/\{r_1,r_2,r_3,s_2\}$ contains an $M(K_5)$-minor.\end{center}
{\tt \scriptsize {\bf Command:} ./macek -pGF2 '!contract 1;!contract 2;!contract 3;!contract -2;!minor' g3 '\{grK5,grK33\}'}\\
{\tt \scriptsize {\bf Output: }The \#1 matroid [g3$\sim$c1$\sim$c2$\sim$c3$\sim$c-2] +HAS+ minor \#1 [grK5] in the list \{grK5 grK33\}.}

\subsubsection*{The matroid $M(G_4)$:}
\FloatBarrier
\begin{figure*}[h]
\raisebox{-10ex}{\resizebox*{0.6\width}{!}{
\psfrag{r1}{\footnotesize $r_1$}
\psfrag{r2}{\footnotesize $r_2$}
\psfrag{r3}{\footnotesize $r_3$}
\psfrag{r4}{\footnotesize $r_4$}
\psfrag{r5}{\footnotesize $r_5$}
\psfrag{r6}{\footnotesize $r_6$}
\psfrag{r7}{\footnotesize $r_7$}
\psfrag{r8}{\footnotesize $r_8$}
\psfrag{r9}{\footnotesize $r_9$}
\psfrag{s1}{\footnotesize $s_1$}
\psfrag{s2}{\footnotesize $s_2$}
\psfrag{s3}{\footnotesize $s_3$}
\psfrag{s4}{\footnotesize $s_4$}
\psfrag{s5}{\footnotesize $s_5$}
\psfrag{s6}{\footnotesize $s_6$}
\psfrag{s7}{\footnotesize $s_7$}
\psfrag{s8}{\footnotesize $s_8$}
\psfrag{s9}{\footnotesize $s_9$}
\includegraphics [scale=0.8] {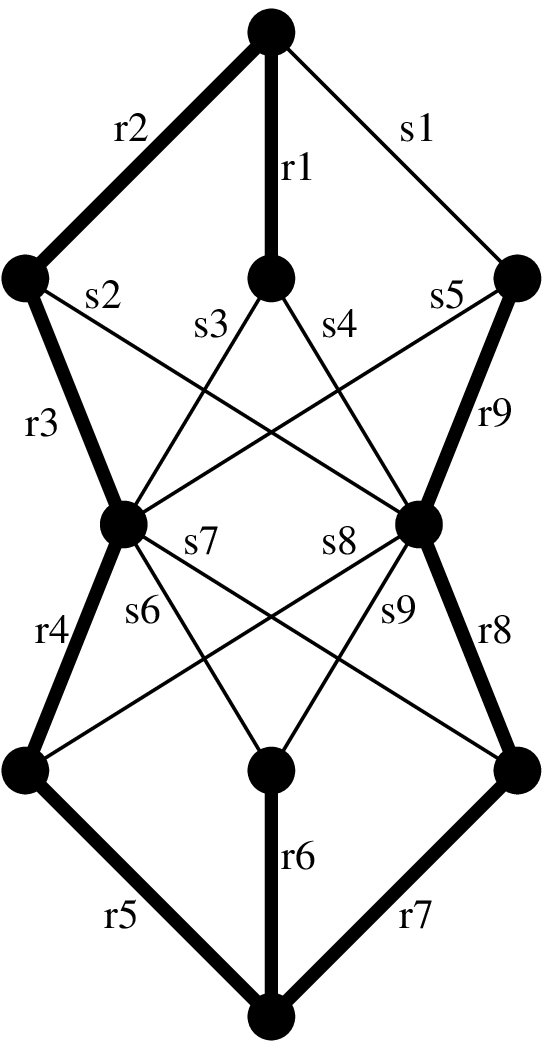}}}
$\mspace{30mu}$
$
g_4=
\kbordermatrix{\mbox{}&s_1&s_2&s_3&s_4&s_5&s_6&s_7&s_8&s_9\\
r_1 & {\;\,\! 0}  	   & {\;\,\! 0}          & {\;\,\! 1}          & {\;\,\! 1}          & {\;\,\! 0}          & {\;\,\! 0}          & {\;\,\! 0}         & {\;\,\! 0}          & {\;\,\! 0}    \\
r_2 & {\;\,\! 1}           & {\;\,\! 0}          & {\;\,\! 1}          & {\;\,\! 1}          & {\;\,\! 0}          & {\;\,\! 0}          & {\;\,\! 0}         & {\;\,\! 0}          & {\;\,\! 0}    \\ 
r_3 & {\;\,\! 1}           & {\;\,\! 1}          & {\;\,\! 1}          & {\;\,\! 1}          & {\;\,\! 0}          & {\;\,\! 0}          & {\;\,\! 0}         & {\;\,\! 0}          & {\;\,\! 0}     \\ 
r_4 & {\;\,\! 1}           & {\;\,\! 1}          & {\;\,\! 0}          & {\;\,\! 1}          & {\;\,\! 1}          & {\;\,\! 1}          & {\;\,\! 1}         & {\;\,\! 0}          & {\;\,\! 0}    \\ 
r_5 & {\;\,\! 1}           & {\;\,\! 1}          & {\;\,\! 0}          & {\;\,\! 1}          & {\;\,\! 1}          & {\;\,\! 1}          & {\;\,\! 1}         & {\;\,\! 1}          & {\;\,\! 0}     \\ 
r_6 & {\;\,\! 0}           & {\;\,\! 0}          & {\;\,\! 0}          & {\;\,\! 0}          & {\;\,\! 0}          & {\;\,\! 1}          & {\;\,\! 0}         & {\;\,\! 0}          & {\;\,\! 1}     \\ 
r_7 & {\;\,\! 1}           & {\;\,\! 1}          & {\;\,\! 0}          & {\;\,\! 1}          & {\;\,\! 1}          & {\;\,\! 0}          & {\;\,\! 1}         & {\;\,\! 1}          & {\;\,\! 1}     \\
r_8 & {\;\,\! 1}           & {\;\,\! 1}          & {\;\,\! 0}          & {\;\,\! 1}          & {\;\,\! 1}          & {\;\,\! 0}          & {\;\,\! 1}         & {\;\,\! 1}          & {\;\,\! 1}     \\
r_9 & {\;\,\! 1}           & {\;\,\! 0}          & {\;\,\! 0}          & {\;\,\! 0}          & {\;\,\! 1}          & {\;\,\! 0}          & {\;\,\! 0}         & {\;\,\! 0}          & {\;\,\! 0}      
}   
$
\label{fig_g4}
\end{figure*}
\noindent
\begin{center} $M(G_4)/\{r_4,r_5,r_6,s_6\}$ contains an $M(K_{3,3})$-minor.\end{center}
{\tt \scriptsize {\bf Command:} ./macek -pGF2 '!contract 4;!contract 5;!contract 6;!contract -6;!minor' g4 '\{grK5,grK33\}'}\\
{\tt \scriptsize {\bf Output: }The \#1 matroid [g4$\sim$c4$\sim$c5$\sim$c6$\sim$c-6] +HAS+ minor \#2 [grK33] in the list \{grK5 grK33\}.}

\subsubsection*{The matroid $M(G_5)$:}
\FloatBarrier
\begin{figure*}[h]
\raisebox{-10ex}{\resizebox*{0.6\width}{!}{
\psfrag{r1}{\footnotesize $r_1$}
\psfrag{r2}{\footnotesize $r_2$}
\psfrag{r3}{\footnotesize $r_3$}
\psfrag{r4}{\footnotesize $r_4$}
\psfrag{r5}{\footnotesize $r_5$}
\psfrag{r6}{\footnotesize $r_6$}
\psfrag{r7}{\footnotesize $r_7$}
\psfrag{r8}{\footnotesize $r_8$}
\psfrag{r9}{\footnotesize $r_9$}
\psfrag{s1}{\footnotesize $s_1$}
\psfrag{s2}{\footnotesize $s_2$}
\psfrag{s3}{\footnotesize $s_3$}
\psfrag{s4}{\footnotesize $s_4$}
\psfrag{s5}{\footnotesize $s_5$}
\psfrag{s6}{\footnotesize $s_6$}
\psfrag{s7}{\footnotesize $s_7$}
\psfrag{s8}{\footnotesize $s_8$}
\includegraphics [scale=0.8] {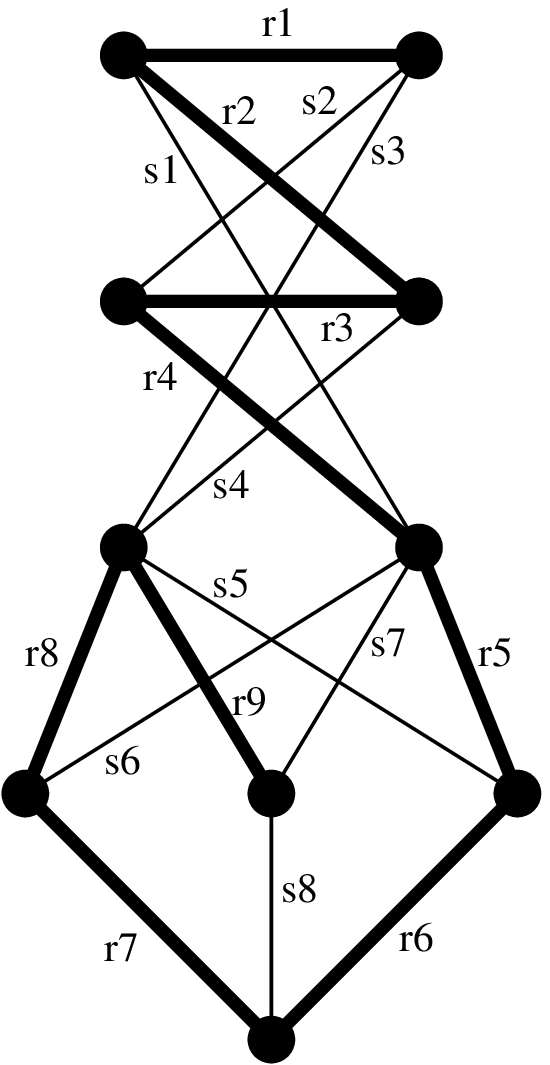}}}
$\mspace{30mu}$
$
g_5=
\kbordermatrix{\mbox{}&s_1&s_2&s_3&s_4&s_5&s_6&s_7&s_8\\
r_1 & {\;\,\! 0}           & {\;\,\! 1}          & {\;\,\! 1}          & {\;\,\! 0}          & {\;\,\! 0}          & {\;\,\! 0}          & {\;\,\! 0}         & {\;\,\! 0}          \\
r_2 & {\;\,\! 1}           & {\;\,\! 1}          & {\;\,\! 1}          & {\;\,\! 0}          & {\;\,\! 0}          & {\;\,\! 0}          & {\;\,\! 0}         & {\;\,\! 0}          \\ 
r_3 & {\;\,\! 1}           & {\;\,\! 1}          & {\;\,\! 1}          & {\;\,\! 1}          & {\;\,\! 0}          & {\;\,\! 0}          & {\;\,\! 0}         & {\;\,\! 0}            \\ 
r_4 & {\;\,\! 1}           & {\;\,\! 0}          & {\;\,\! 1}          & {\;\,\! 1}          & {\;\,\! 0}          & {\;\,\! 0}          & {\;\,\! 0}         & {\;\,\! 0}          \\ 
r_5 & {\;\,\! 0}           & {\;\,\! 0}          & {\;\,\! 1}          & {\;\,\! 1}          & {\;\,\! 0}          & {\;\,\! 1}          & {\;\,\! 1}         & {\;\,\! 0}          \\ 
r_6 & {\;\,\! 0}           & {\;\,\! 0}          & {\;\,\! 1}          & {\;\,\! 1}          & {\;\,\! 1}          & {\;\,\! 1}          & {\;\,\! 1}         & {\;\,\! 0}           \\ 
r_7 & {\;\,\! 0}           & {\;\,\! 0}          & {\;\,\! 1}          & {\;\,\! 1}          & {\;\,\! 1}          & {\;\,\! 1}          & {\;\,\! 1}         & {\;\,\! 1}            \\
r_8 & {\;\,\! 0}           & {\;\,\! 0}          & {\;\,\! 1}          & {\;\,\! 1}          & {\;\,\! 1}          & {\;\,\! 0}          & {\;\,\! 1}         & {\;\,\! 1}           \\
r_9 & {\;\,\! 0}           & {\;\,\! 0}          & {\;\,\! 0}          & {\;\,\! 0}          & {\;\,\! 0}          & {\;\,\! 0}          & {\;\,\! 1}         & {\;\,\! 1}               
}   
$
\label{fig_g5}
\end{figure*}
\noindent
\begin{center} $M(G_5)/\{r_7,r_8,r_9,s_8\}$ contains an $M(K_{3,3})$-minor.\end{center}
{\tt \scriptsize {\bf Command:}./macek -pGF2 '!contract 7;!contract 8;!contract 9;!contract -8;!minor' g5 '\{grK5,grK33\}'}\\
{\tt \scriptsize {\bf Output: }The \#1 matroid [g5$\sim$c7$\sim$c8$\sim$c9$\sim$c-8] +HAS+ minor \#2 [grK33] in the list \{grK5 grK33\}.}

\subsubsection*{The matroid $M(G_6)$:}
\FloatBarrier
\begin{figure*}[h]
\raisebox{-10ex}{\resizebox*{0.6\width}{!}{
\psfrag{r1}{\footnotesize $r_1$}
\psfrag{r2}{\footnotesize $r_2$}
\psfrag{r3}{\footnotesize $r_3$}
\psfrag{r4}{\footnotesize $r_4$}
\psfrag{r5}{\footnotesize $r_5$}
\psfrag{r6}{\footnotesize $r_6$}
\psfrag{r7}{\footnotesize $r_7$}
\psfrag{r8}{\footnotesize $r_8$}
\psfrag{r9}{\footnotesize $r_9$}
\psfrag{s1}{\footnotesize $s_1$}
\psfrag{s2}{\footnotesize $s_2$}
\psfrag{s3}{\footnotesize $s_3$}
\psfrag{s4}{\footnotesize $s_4$}
\psfrag{s5}{\footnotesize $s_5$}
\psfrag{s6}{\footnotesize $s_6$}
\psfrag{s7}{\footnotesize $s_7$}
\includegraphics [scale=0.8] {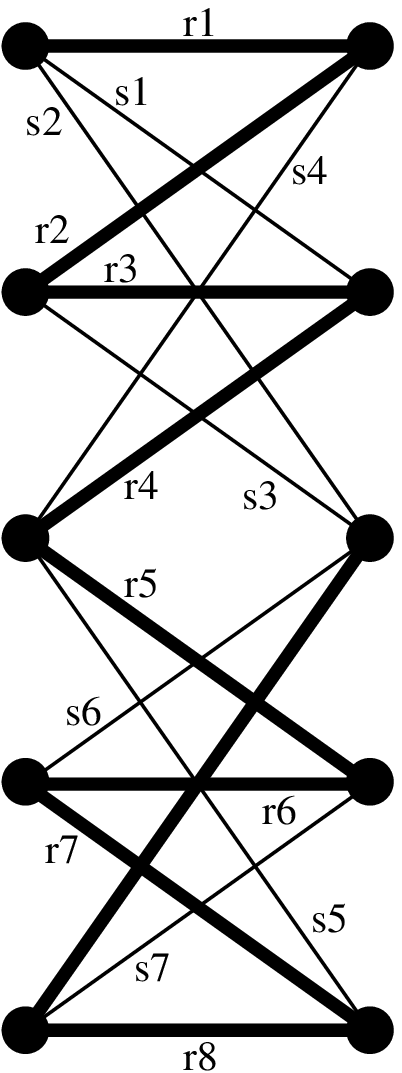}}}
$\mspace{30mu}$
$
g_6=
\kbordermatrix{\mbox{}&s_1&s_2&s_3&s_4&s_5&s_6&s_7\\
r_1 & {\;\,\! 1}  	   & {\;\,\! 1}          & {\;\,\! 0}          & {\;\,\! 0}          & {\;\,\! 0}          & {\;\,\! 0}          & {\;\,\! 0}            \\
r_2 & {\;\,\! 1}           & {\;\,\! 1}          & {\;\,\! 0}          & {\;\,\! 1}          & {\;\,\! 0}          & {\;\,\! 0}          & {\;\,\! 0}            \\ 
r_3 & {\;\,\! 1}           & {\;\,\! 1}          & {\;\,\! 1}          & {\;\,\! 1}          & {\;\,\! 0}          & {\;\,\! 0}          & {\;\,\! 0}             \\ 
r_4 & {\;\,\! 0}           & {\;\,\! 1}          & {\;\,\! 1}          & {\;\,\! 1}          & {\;\,\! 0}          & {\;\,\! 0}          & {\;\,\! 0}             \\ 
r_5 & {\;\,\! 0}           & {\;\,\! 1}          & {\;\,\! 1}          & {\;\,\! 0}          & {\;\,\! 1}          & {\;\,\! 0}          & {\;\,\! 0}            \\ 
r_6 & {\;\,\! 0}           & {\;\,\! 1}          & {\;\,\! 1}          & {\;\,\! 0}          & {\;\,\! 1}          & {\;\,\! 0}          & {\;\,\! 1}              \\ 
r_7 & {\;\,\! 0}           & {\;\,\! 1}          & {\;\,\! 1}          & {\;\,\! 0}          & {\;\,\! 1}          & {\;\,\! 1}          & {\;\,\! 1}              \\
r_8 & {\;\,\! 0}           & {\;\,\! 1}          & {\;\,\! 1}          & {\;\,\! 0}          & {\;\,\! 0}          & {\;\,\! 1}          & {\;\,\! 1}              \\
r_9 & {\;\,\! 0}           & {\;\,\! 1}          & {\;\,\! 1}          & {\;\,\! 0}          & {\;\,\! 0}          & {\;\,\! 1}          & {\;\,\! 0}                  
}   
$
\label{fig_g6}
\end{figure*}
\noindent
\begin{center} $M(G_6)/\{r_6,r_7,r_8,s_7\}$ contains an $M(K_{3,3})$-minor.\end{center}
{\tt \scriptsize {\bf Command:} ./macek -pGF2 '!contract 6;!contract 7;!contract 8;!contract -7;!minor' g6 '\{grK5,grK33\}'}\\
Output:  {\tt \scriptsize {\bf Output: } The \#1 matroid [g6$\sim$c6$\sim$c7$\sim$c8$\sim$c-7] +HAS+ minor \#2 [grK33] in the list \{grK5 grK33\}.}

\subsubsection*{The matroid $M(G_7)$:}
\FloatBarrier
\begin{figure*}[h]
\raisebox{-10ex}{\resizebox*{0.6\width}{!}{
\psfrag{r1}{\footnotesize $r_1$}
\psfrag{r2}{\footnotesize $r_2$}
\psfrag{r3}{\footnotesize $r_3$}
\psfrag{r4}{\footnotesize $r_4$}
\psfrag{r5}{\footnotesize $r_5$}
\psfrag{r6}{\footnotesize $r_6$}
\psfrag{s1}{\footnotesize $s_1$}
\psfrag{s2}{\footnotesize $s_2$}
\psfrag{s3}{\footnotesize $s_3$}
\psfrag{s4}{\footnotesize $s_4$}
\psfrag{s5}{\footnotesize $s_5$}
\psfrag{s6}{\footnotesize $s_6$}
\psfrag{s7}{\footnotesize $s_7$}
\psfrag{s8}{\footnotesize $s_8$}
\psfrag{s9}{\footnotesize $s_9$}
\psfrag{s10}{\footnotesize $s_{10}$}
\psfrag{s11}{\footnotesize $s_{11}$}
\includegraphics [scale=0.8] {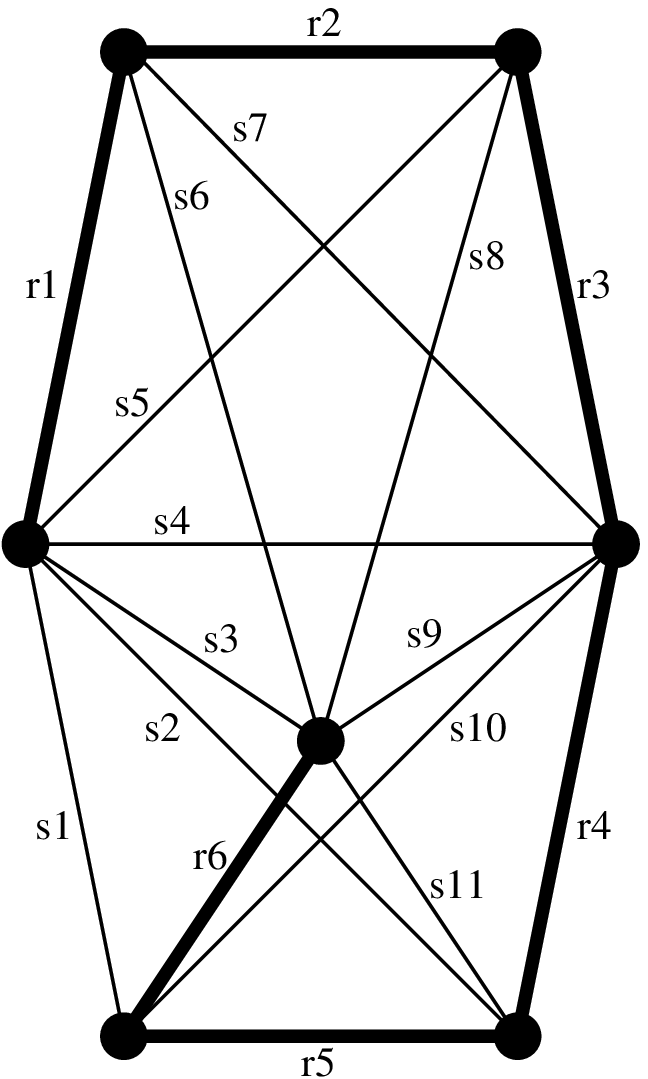}}}
$\mspace{30mu}$
$
g_7=
\kbordermatrix{\mbox{}&s_1&s_2&s_3&s_4&s_5&s_6&s_7&s_8&s_9&s_{10}&s_{11}\\
r_1 & {\;\,\! 1}  	   & {\;\,\! 1}          & {\;\,\! 1}          & {\;\,\! 1}          & {\;\,\! 1}          & {\;\,\! 0}          & {\;\,\! 1}           & {\;\,\! 0}       & {\;\,\! 0}      & {\;\,\! 0}       & {\;\,\! 0}\\
r_2 & {\;\,\! 1}           & {\;\,\! 1}          & {\;\,\! 1}          & {\;\,\! 1}          & {\;\,\! 1}          & {\;\,\! 1}          & {\;\,\! 1}           & {\;\,\! 0}       & {\;\,\! 0}      & {\;\,\! 0}       & {\;\,\! 0}\\ 
r_3 & {\;\,\! 1}           & {\;\,\! 1}          & {\;\,\! 1}          & {\;\,\! 1}          & {\;\,\! 0}          & {\;\,\! 1}          & {\;\,\! 0}           & {\;\,\! 1}       & {\;\,\! 0}      & {\;\,\! 0}       & {\;\,\! 0}\\ 
r_4 & {\;\,\! 1}           & {\;\,\! 1}          & {\;\,\! 1}          & {\;\,\! 0}          & {\;\,\! 0}          & {\;\,\! 1}          & {\;\,\! 0}           & {\;\,\! 1}       & {\;\,\! 1}      & {\;\,\! 1}       & {\;\,\! 0}\\
r_5 & {\;\,\! 1}           & {\;\,\! 0}          & {\;\,\! 1}          & {\;\,\! 0}          & {\;\,\! 0}          & {\;\,\! 1}          & {\;\,\! 0}           & {\;\,\! 1}       & {\;\,\! 1}      & {\;\,\! 1}       & {\;\,\! 1}\\ 
r_6 & {\;\,\! 0}           & {\;\,\! 0}          & {\;\,\! 1}          & {\;\,\! 0}          & {\;\,\! 0}          & {\;\,\! 1}          & {\;\,\! 0}           & {\;\,\! 1}       & {\;\,\! 1}      & {\;\,\! 0}       & {\;\,\! 1}          
}   
$
\label{fig_g7}
\end{figure*}
\noindent
\begin{center} $M(G_7)/\{r_1,r_2,s_5\}$ contains an $M(K_5)$-minor.\end{center}
{\tt \scriptsize {\bf Command:} ./macek -pGF2 '!contract 1;!contract 2;!contract -5;!minor' g7 '\{grK5,grK33\}'}\\
{\tt \scriptsize {\bf Output: }The \#1 matroid [g7$\sim$c1$\sim$c2$\sim$c-5] +HAS+ minor \#1 [grK5] in the list \{grK5 grK33\}.}

\subsubsection*{The matroid $M(G_8)$:}
\FloatBarrier
\begin{figure*}[h]
\raisebox{-10ex}{\resizebox*{0.6\width}{!}{
\psfrag{r1}{\footnotesize $r_1$}
\psfrag{r2}{\footnotesize $r_2$}
\psfrag{r3}{\footnotesize $r_3$}
\psfrag{r4}{\footnotesize $r_4$}
\psfrag{r5}{\footnotesize $r_5$}
\psfrag{r6}{\footnotesize $r_6$}
\psfrag{r7}{\footnotesize $r_7$}
\psfrag{s1}{\footnotesize $s_1$}
\psfrag{s2}{\footnotesize $s_2$}
\psfrag{s3}{\footnotesize $s_3$}
\psfrag{s4}{\footnotesize $s_4$}
\psfrag{s5}{\footnotesize $s_5$}
\psfrag{s6}{\footnotesize $s_6$}
\psfrag{s7}{\footnotesize $s_7$}
\psfrag{s8}{\footnotesize $s_8$}
\psfrag{s9}{\footnotesize $s_9$}
\psfrag{s10}{\footnotesize $s_{10}$}
\includegraphics [scale=0.8] {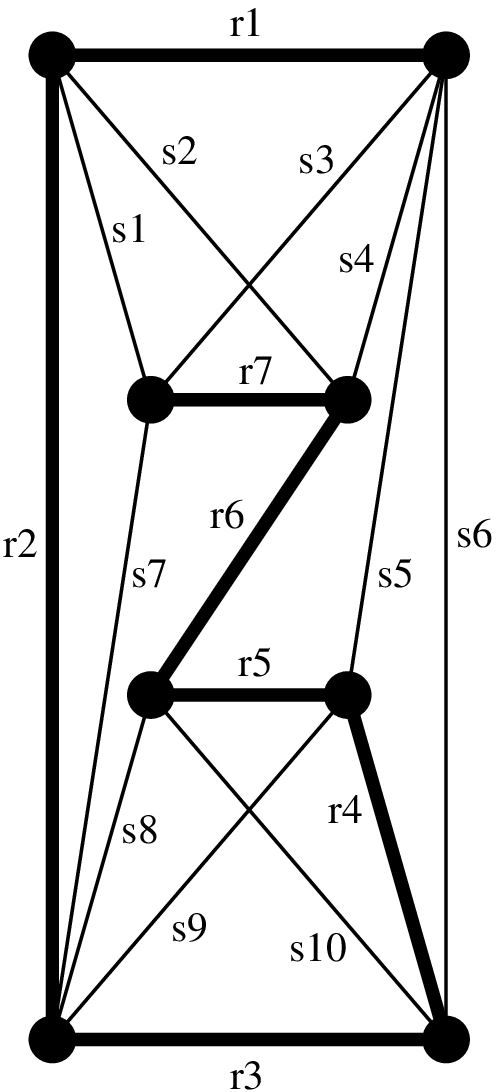}}}
$\mspace{30mu}$
$
g_8=
\kbordermatrix{\mbox{}&s_1&s_2&s_3&s_4&s_5&s_6&s_7&s_8&s_9&s_{10}\\
r_1 & {\;\,\! 1}  	   & {\;\,\! 1}          & {\;\,\! 1}          & {\;\,\! 1}          & {\;\,\! 1}          & {\;\,\! 0}          & {\;\,\! 1}           & {\;\,\! 0}       & {\;\,\! 0}      & {\;\,\! 0}   \\
r_2 & {\;\,\! 1}           & {\;\,\! 1}          & {\;\,\! 1}          & {\;\,\! 1}          & {\;\,\! 1}          & {\;\,\! 1}          & {\;\,\! 1}           & {\;\,\! 0}       & {\;\,\! 0}      & {\;\,\! 0}    \\ 
r_3 & {\;\,\! 1}           & {\;\,\! 1}          & {\;\,\! 1}          & {\;\,\! 1}          & {\;\,\! 0}          & {\;\,\! 1}          & {\;\,\! 0}           & {\;\,\! 1}       & {\;\,\! 0}      & {\;\,\! 0}     \\ 
r_4 & {\;\,\! 1}           & {\;\,\! 1}          & {\;\,\! 1}          & {\;\,\! 0}          & {\;\,\! 0}          & {\;\,\! 1}          & {\;\,\! 0}           & {\;\,\! 1}       & {\;\,\! 1}      & {\;\,\! 1}     \\
r_5 & {\;\,\! 1}           & {\;\,\! 0}          & {\;\,\! 1}          & {\;\,\! 0}          & {\;\,\! 0}          & {\;\,\! 1}          & {\;\,\! 0}           & {\;\,\! 1}       & {\;\,\! 1}      & {\;\,\! 1}     \\ 
r_6 & {\;\,\! 0}           & {\;\,\! 0}          & {\;\,\! 1}          & {\;\,\! 0}          & {\;\,\! 0}          & {\;\,\! 1}          & {\;\,\! 0}           & {\;\,\! 1}       & {\;\,\! 1}      & {\;\,\! 0}     \\          
r_7 & {\;\,\! 0}           & {\;\,\! 0}          & {\;\,\! 1}          & {\;\,\! 0}          & {\;\,\! 0}          & {\;\,\! 1}          & {\;\,\! 0}           & {\;\,\! 1}       & {\;\,\! 1}      & {\;\,\! 0}       
}
$
\label{fig_g8}
\end{figure*}
\noindent
\begin{center} $M(G_8)/\{r_4,s_5,s_6\}$ contains an $M(K_5)$-minor.\end{center}
 {\tt \scriptsize {\bf Command:} ./macek -pGF2 '!contract 4;!contract -5;!contract -6;!minor' g8 '\{grK5,grK33\}'}\\
{\tt \scriptsize {\bf Output: }
The \#1 matroid [g8$\sim$c4$\sim$c-5$\sim$c-6] +HAS+ minor \#1 [grK5] in the list \{grK5 grK33\}.}
\FloatBarrier
\subsubsection*{The matroid $M(G_9)$:}
\FloatBarrier
\begin{figure*}[h]
\raisebox{-10ex}{\resizebox*{0.6\width}{!}{
\psfrag{r1}{\footnotesize $r_1$}
\psfrag{r2}{\footnotesize $r_2$}
\psfrag{r3}{\footnotesize $r_3$}
\psfrag{r4}{\footnotesize $r_4$}
\psfrag{r5}{\footnotesize $r_5$}
\psfrag{r6}{\footnotesize $r_6$}
\psfrag{r7}{\footnotesize $r_7$}
\psfrag{r8}{\footnotesize $r_8$}
\psfrag{s1}{\footnotesize $s_1$}
\psfrag{s2}{\footnotesize $s_2$}
\psfrag{s3}{\footnotesize $s_3$}
\psfrag{s4}{\footnotesize $s_4$}
\psfrag{s5}{\footnotesize $s_5$}
\psfrag{s6}{\footnotesize $s_6$}
\psfrag{s7}{\footnotesize $s_7$}
\psfrag{s8}{\footnotesize $s_8$}
\includegraphics [scale=0.8] {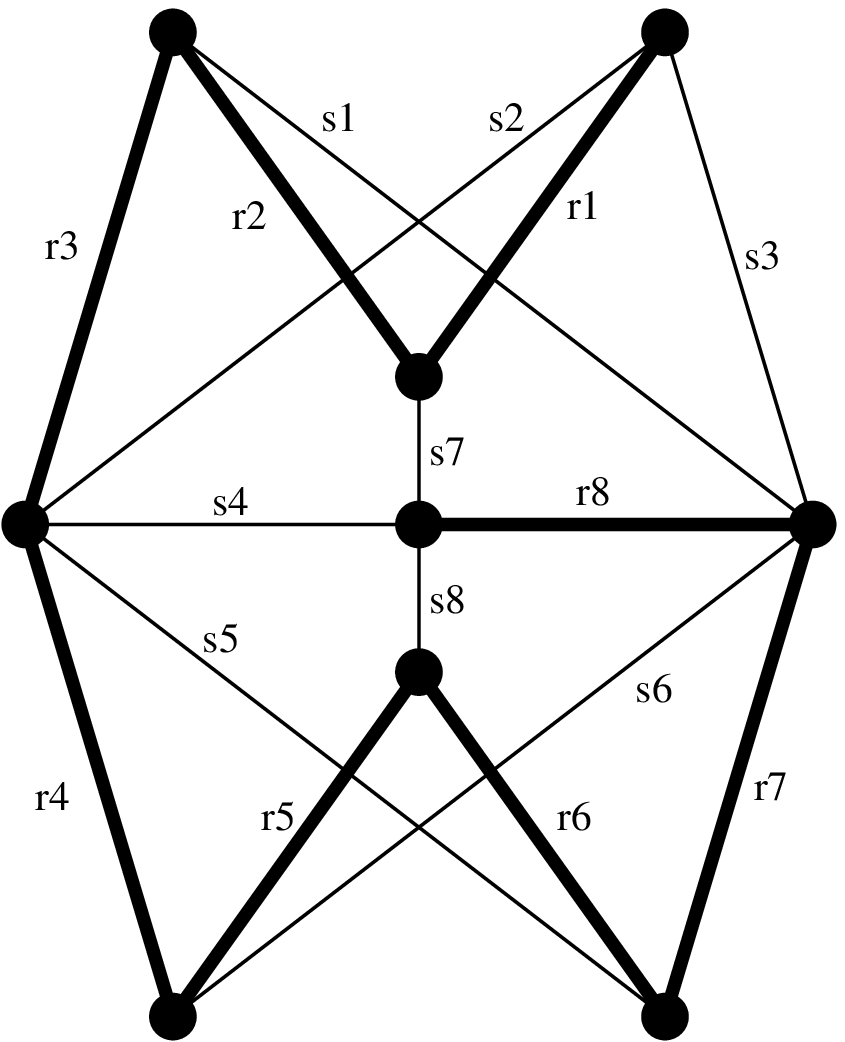}}}
$\mspace{30mu}$
$
g_9=
\kbordermatrix{\mbox{}&s_1&s_2&s_3&s_4&s_5&s_6&s_7&s_8\\
r_1 & {\;\,\! 0}  	   & {\;\,\! 1}          & {\;\,\! 1}          & {\;\,\! 0}          & {\;\,\! 0}          & {\;\,\! 0}          & {\;\,\! 0}           & {\;\,\! 0}      \\
r_2 & {\;\,\! 0}           & {\;\,\! 1}          & {\;\,\! 1}          & {\;\,\! 0}          & {\;\,\! 0}          & {\;\,\! 0}          & {\;\,\! 1}           & {\;\,\! 0}      \\ 
r_3 & {\;\,\! 1}           & {\;\,\! 1}          & {\;\,\! 1}          & {\;\,\! 0}          & {\;\,\! 0}          & {\;\,\! 0}          & {\;\,\! 1}           & {\;\,\! 0}       \\ 
r_4 & {\;\,\! 1}           & {\;\,\! 0}          & {\;\,\! 1}          & {\;\,\! 1}          & {\;\,\! 1}          & {\;\,\! 0}          & {\;\,\! 1}           & {\;\,\! 0}        \\
r_5 & {\;\,\! 1}           & {\;\,\! 0}          & {\;\,\! 1}          & {\;\,\! 1}          & {\;\,\! 1}          & {\;\,\! 1}          & {\;\,\! 1}           & {\;\,\! 0}         \\ 
r_6 & {\;\,\! 1}           & {\;\,\! 0}          & {\;\,\! 1}          & {\;\,\! 1}          & {\;\,\! 1}          & {\;\,\! 1}          & {\;\,\! 1}           & {\;\,\! 1}       \\          
r_7 & {\;\,\! 1}           & {\;\,\! 0}          & {\;\,\! 1}          & {\;\,\! 1}          & {\;\,\! 0}          & {\;\,\! 1}          & {\;\,\! 1}           & {\;\,\! 1}     \\
r_8 & {\;\,\! 0}           & {\;\,\! 0}          & {\;\,\! 0}          & {\;\,\! 1}          & {\;\,\! 0}          & {\;\,\! 0}          & {\;\,\! 1}           & {\;\,\! 1}            
}
$
\label{fig_g9}
\end{figure*}
\noindent
\begin{center} $M(G_9)/\{r_1,r_2,r_3,s_2\}$ contains an $M(K_{3,3})$-minor.\end{center}
{\tt \scriptsize {\bf Command:} ./macek -pGF2 '!contract 1;!contract 2;!contract 3;!contract -2;!minor' g9 '\{grK5,grK33\}'}\\
{\tt \scriptsize {\bf Output: }
The \#1 matroid [g9$\sim$c1$\sim$c2$\sim$c3$\sim$c-2] +HAS+ minor \#2 [grK33] in the list \{grK5 grK33\}.}

\subsubsection*{The matroid $M(G_{10})$:}
\FloatBarrier
\begin{figure*}[h]
\raisebox{-10ex}{\resizebox*{0.6\width}{!}{
\psfrag{r1}{\footnotesize $r_1$}
\psfrag{r2}{\footnotesize $r_2$}
\psfrag{r3}{\footnotesize $r_3$}
\psfrag{r4}{\footnotesize $r_4$}
\psfrag{r5}{\footnotesize $r_5$}
\psfrag{r6}{\footnotesize $r_6$}
\psfrag{r7}{\footnotesize $r_7$}
\psfrag{s1}{\footnotesize $s_1$}
\psfrag{s2}{\footnotesize $s_2$}
\psfrag{s3}{\footnotesize $s_3$}
\psfrag{s4}{\footnotesize $s_4$}
\psfrag{s5}{\footnotesize $s_5$}
\psfrag{s6}{\footnotesize $s_6$}
\psfrag{s7}{\footnotesize $s_7$}
\psfrag{s8}{\footnotesize $s_8$}
\psfrag{s9}{\footnotesize $s_9$}
\includegraphics [scale=0.8] {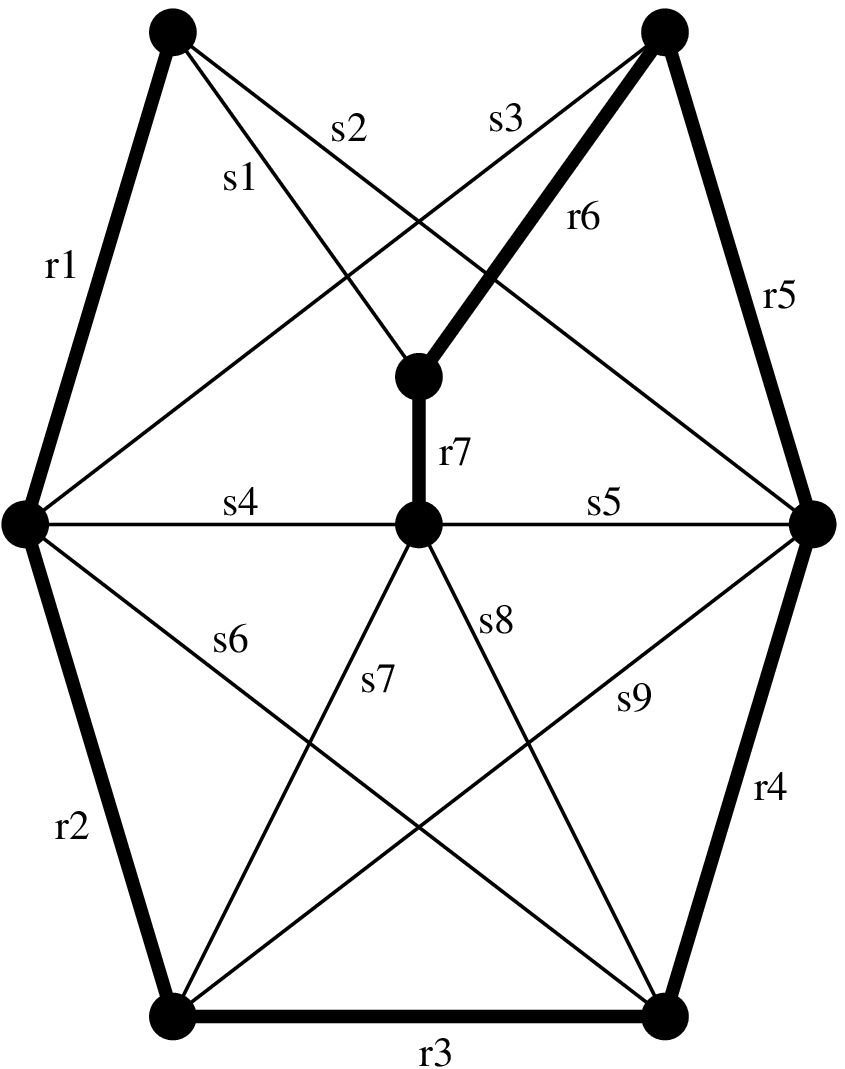}}}
$\mspace{30mu}$
$
g_{10}=
\kbordermatrix{\mbox{}&s_1&s_2&s_3&s_4&s_5&s_6&s_7&s_8&s_9\\
r_1 & {\;\,\! 1}  	   & {\;\,\! 1}          & {\;\,\! 0}          & {\;\,\! 0}          & {\;\,\! 0}          & {\;\,\! 0}          & {\;\,\! 0}           & {\;\,\! 0}     & {\;\,\! 0} \\
r_2 & {\;\,\! 1}           & {\;\,\! 1}          & {\;\,\! 1}          & {\;\,\! 1}          & {\;\,\! 0}          & {\;\,\! 1}          & {\;\,\! 0}           & {\;\,\! 0}     & {\;\,\! 0} \\ 
r_3 & {\;\,\! 1}           & {\;\,\! 1}          & {\;\,\! 1}          & {\;\,\! 1}          & {\;\,\! 0}          & {\;\,\! 1}          & {\;\,\! 1}           & {\;\,\! 0}     & {\;\,\! 1} \\ 
r_4 & {\;\,\! 1}           & {\;\,\! 1}          & {\;\,\! 1}          & {\;\,\! 1}          & {\;\,\! 0}          & {\;\,\! 0}          & {\;\,\! 1}           & {\;\,\! 1}     & {\;\,\! 1}\\
r_5 & {\;\,\! 1}           & {\;\,\! 0}          & {\;\,\! 1}          & {\;\,\! 1}          & {\;\,\! 1}          & {\;\,\! 0}          & {\;\,\! 1}           & {\;\,\! 1}     & {\;\,\! 0}    \\ 
r_6 & {\;\,\! 1}           & {\;\,\! 0}          & {\;\,\! 0}          & {\;\,\! 1}          & {\;\,\! 1}          & {\;\,\! 0}          & {\;\,\! 1}           & {\;\,\! 1}     & {\;\,\! 0}\\          
r_7 & {\;\,\! 0}           & {\;\,\! 0}          & {\;\,\! 0}          & {\;\,\! 1}          & {\;\,\! 1}          & {\;\,\! 0}          & {\;\,\! 1}           & {\;\,\! 1}     & {\;\,\! 0}
}
$
\label{fig_g10}
\end{figure*}
\noindent
\begin{center} $M(G_{10})/\{r_1,r_6,s_1,s_3\}$ contains an $M(K_5)$-minor.\end{center}
{\tt \scriptsize {\bf Command:} ./macek -pGF2 '!contract 1;!contract 6;!contract -1;!contract -3;!minor' g10 '\{grK5,grK33\}'}\\
{\tt \scriptsize {\bf Output: }
The \#1 matroid [g10$\sim$c1$\sim$c6$\sim$c-1$\sim$c-3] +HAS+ minor \#1 [grK5] in the list \{grK5 grK33\}.}

\subsubsection*{The matroid $M(G_{11})$:}
\FloatBarrier
\begin{figure*}[h]
\raisebox{-10ex}{\resizebox*{0.6\width}{!}{
\psfrag{r1}{\footnotesize $r_1$}
\psfrag{r2}{\footnotesize $r_2$}
\psfrag{r3}{\footnotesize $r_3$}
\psfrag{r4}{\footnotesize $r_4$}
\psfrag{r5}{\footnotesize $r_5$}
\psfrag{r6}{\footnotesize $r_6$}
\psfrag{r7}{\footnotesize $r_7$}
\psfrag{r8}{\footnotesize $r_8$}
\psfrag{s1}{\footnotesize $s_1$}
\psfrag{s2}{\footnotesize $s_2$}
\psfrag{s3}{\footnotesize $s_3$}
\psfrag{s4}{\footnotesize $s_4$}
\psfrag{s5}{\footnotesize $s_5$}
\psfrag{s6}{\footnotesize $s_6$}
\psfrag{s7}{\footnotesize $s_7$}
\psfrag{s8}{\footnotesize $s_8$}
\includegraphics [scale=0.8] {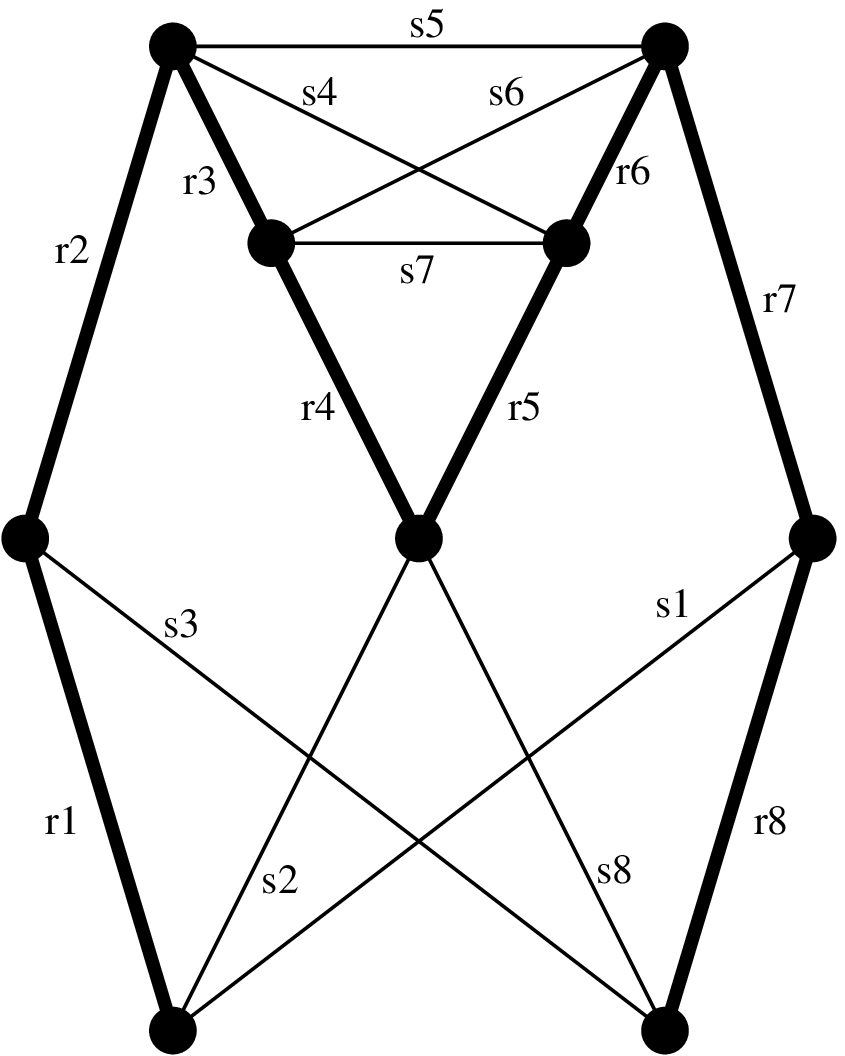}}}
$\mspace{30mu}$
$
g_{11}=
\kbordermatrix{\mbox{}&s_1&s_2&s_3&s_4&s_5&s_6&s_7&s_8\\
r_1 & {\;\,\! 1}  	   & {\;\,\! 1}          & {\;\,\! 0}          & {\;\,\! 0}          & {\;\,\! 0}          & {\;\,\! 0}          & {\;\,\! 0}           & {\;\,\! 0}     \\
r_2 & {\;\,\! 1}           & {\;\,\! 1}          & {\;\,\! 1}          & {\;\,\! 0}          & {\;\,\! 0}          & {\;\,\! 0}          & {\;\,\! 0}           & {\;\,\! 0}     \\ 
r_3 & {\;\,\! 1}           & {\;\,\! 1}          & {\;\,\! 1}          & {\;\,\! 1}          & {\;\,\! 1}          & {\;\,\! 0}          & {\;\,\! 0}           & {\;\,\! 0}      \\ 
r_4 & {\;\,\! 1}           & {\;\,\! 1}          & {\;\,\! 1}          & {\;\,\! 1}          & {\;\,\! 1}          & {\;\,\! 1}          & {\;\,\! 1}           & {\;\,\! 0}     \\
r_5 & {\;\,\! 1}           & {\;\,\! 0}          & {\;\,\! 1}          & {\;\,\! 1}          & {\;\,\! 1}          & {\;\,\! 1}          & {\;\,\! 1}           & {\;\,\! 1}       \\ 
r_6 & {\;\,\! 1}           & {\;\,\! 0}          & {\;\,\! 1}          & {\;\,\! 0}          & {\;\,\! 1}          & {\;\,\! 1}          & {\;\,\! 0}           & {\;\,\! 1}     \\          
r_7 & {\;\,\! 1}           & {\;\,\! 0}          & {\;\,\! 1}          & {\;\,\! 0}          & {\;\,\! 0}          & {\;\,\! 0}          & {\;\,\! 0}           & {\;\,\! 1}    \\
r_8 & {\;\,\! 0}           & {\;\,\! 0}          & {\;\,\! 1}          & {\;\,\! 0}          & {\;\,\! 0}          & {\;\,\! 0}          & {\;\,\! 0}           & {\;\,\! 1}    
}
$
\label{fig_g11}
\end{figure*}
\noindent
\begin{center} $M(G_{11})/\{r_3,r_6,s_5,s_7\}$ contains an $M(K_{3,3})$-minor.\end{center}
{\tt \scriptsize {\bf Command:} ./macek -pGF2 '!contract 3;!contract 6;!contract -5;!contract -7;!minor' g11 '\{grK5,grK33\}'}\\
{\tt \scriptsize {\bf Output: }
The \#1 matroid [g11$\sim$c3$\sim$c6$\sim$c-5$\sim$c-7] +HAS+ minor \#2 [grK33] in the list \{grK5 grK33\}.}

\subsubsection*{The matroid $M(G_{12})$:}
\FloatBarrier
\begin{figure*}[h]
\raisebox{-10ex}{\resizebox*{0.6\width}{!}{
\psfrag{r1}{\footnotesize $r_1$}
\psfrag{r2}{\footnotesize $r_2$}
\psfrag{r3}{\footnotesize $r_3$}
\psfrag{r4}{\footnotesize $r_4$}
\psfrag{r5}{\footnotesize $r_5$}
\psfrag{r6}{\footnotesize $r_6$}
\psfrag{r7}{\footnotesize $r_7$}
\psfrag{r8}{\footnotesize $r_8$}
\psfrag{s1}{\footnotesize $s_1$}
\psfrag{s2}{\footnotesize $s_2$}
\psfrag{s3}{\footnotesize $s_3$}
\psfrag{s4}{\footnotesize $s_4$}
\psfrag{s5}{\footnotesize $s_5$}
\psfrag{s6}{\footnotesize $s_6$}
\psfrag{s7}{\footnotesize $s_7$}
\includegraphics [scale=0.8] {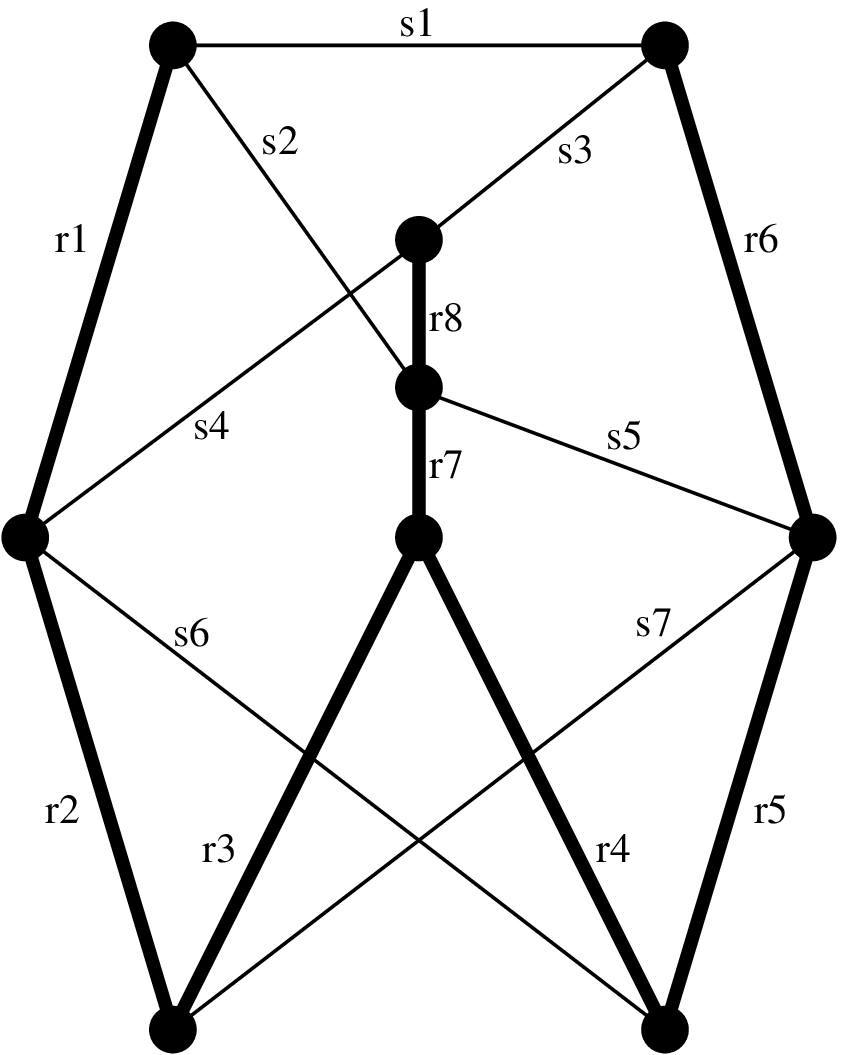}}}
$\mspace{30mu}$
$
g_{12}=
\kbordermatrix{\mbox{}&s_1&s_2&s_3&s_4&s_5&s_6&s_7\\
r_1 & {\;\,\! 1}  	   & {\;\,\! 1}          & {\;\,\! 0}          & {\;\,\! 0}          & {\;\,\! 0}          & {\;\,\! 0}          & {\;\,\! 0}               \\
r_2 & {\;\,\! 1}           & {\;\,\! 1}          & {\;\,\! 0}          & {\;\,\! 1}          & {\;\,\! 0}          & {\;\,\! 1}          & {\;\,\! 0}                \\ 
r_3 & {\;\,\! 1}           & {\;\,\! 1}          & {\;\,\! 0}          & {\;\,\! 1}          & {\;\,\! 0}          & {\;\,\! 1}          & {\;\,\! 1}                 \\ 
r_4 & {\;\,\! 1}           & {\;\,\! 0}          & {\;\,\! 1}          & {\;\,\! 0}          & {\;\,\! 1}          & {\;\,\! 1}          & {\;\,\! 1}               \\
r_5 & {\;\,\! 1}           & {\;\,\! 0}          & {\;\,\! 1}          & {\;\,\! 0}          & {\;\,\! 1}          & {\;\,\! 0}          & {\;\,\! 1}                 \\ 
r_6 & {\;\,\! 1}           & {\;\,\! 0}          & {\;\,\! 1}          & {\;\,\! 0}          & {\;\,\! 0}          & {\;\,\! 0}          & {\;\,\! 0}               \\          
r_7 & {\;\,\! 0}           & {\;\,\! 1}          & {\;\,\! 1}          & {\;\,\! 1}          & {\;\,\! 1}          & {\;\,\! 0}          & {\;\,\! 0}             \\
r_8 & {\;\,\! 0}           & {\;\,\! 0}          & {\;\,\! 1}          & {\;\,\! 1}          & {\;\,\! 1}          & {\;\,\! 0}          & {\;\,\! 0}             
}
$
\label{fig_g12}
\end{figure*}
\noindent
\begin{center} $M(G_{12})/\{r_8,s_1,s_2,s_3\}$ contains an $M(K_{3,3})$-minor.\end{center}
{\tt \scriptsize {\bf Command:} ./macek -pGF2 '!contract 8;!contract -1;!contract -2;!contract -3;!minor' g12 '\{grK5,grK33\}'}\\
{\tt \scriptsize {\bf Output: }
The \#1 matroid [g12$\sim$c8$\sim$c-1$\sim$c-2$\sim$c-3] +HAS+ minor \#2 [grK33] in the list \{grK5 grK33\}.}

\subsubsection*{The matroid $M(G_{13})$:}
\FloatBarrier
\begin{figure*}[h]
\raisebox{-10ex}{\resizebox*{0.6\width}{!}{
\psfrag{r1}{\footnotesize $r_1$}
\psfrag{r2}{\footnotesize $r_2$}
\psfrag{r3}{\footnotesize $r_3$}
\psfrag{r4}{\footnotesize $r_4$}
\psfrag{r5}{\footnotesize $r_5$}
\psfrag{r6}{\footnotesize $r_6$}
\psfrag{r7}{\footnotesize $r_7$}
\psfrag{r8}{\footnotesize $r_8$}
\psfrag{s1}{\footnotesize $s_1$}
\psfrag{s2}{\footnotesize $s_2$}
\psfrag{s3}{\footnotesize $s_3$}
\psfrag{s4}{\footnotesize $s_4$}
\psfrag{s5}{\footnotesize $s_5$}
\psfrag{s6}{\footnotesize $s_6$}
\psfrag{s7}{\footnotesize $s_7$}
\psfrag{s8}{\footnotesize $s_8$}
\psfrag{s9}{\footnotesize $s_9$}
\includegraphics [scale=0.8] {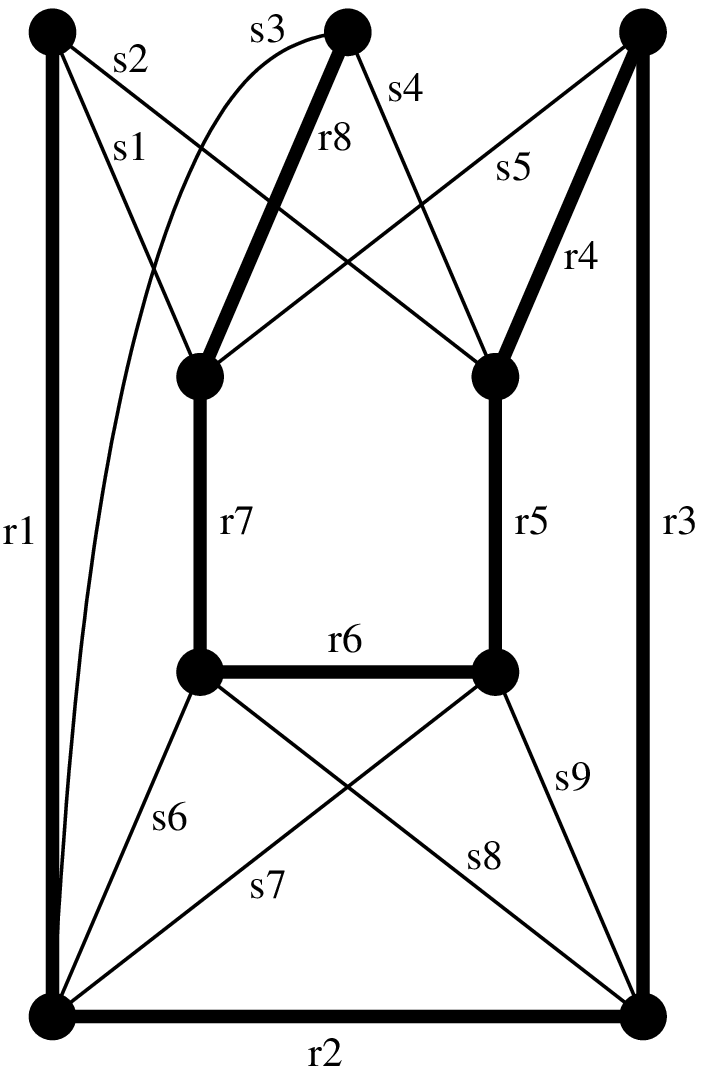}}}
$\mspace{30mu}$
$
g_{13}=
\kbordermatrix{\mbox{}&s_1&s_2&s_3&s_4&s_5&s_6&s_7&s_8&s_9\\
r_1 & {\;\,\! 1}  	   & {\;\,\! 1}          & {\;\,\! 0}          & {\;\,\! 0}          & {\;\,\! 0}          & {\;\,\! 0}          & {\;\,\! 0}           & {\;\,\! 0}            & {\;\,\! 0}\\
r_2 & {\;\,\! 1}           & {\;\,\! 1}          & {\;\,\! 1}          & {\;\,\! 0}          & {\;\,\! 0}          & {\;\,\! 1}          & {\;\,\! 1}           & {\;\,\! 0}            & {\;\,\! 0}\\ 
r_3 & {\;\,\! 1}           & {\;\,\! 1}          & {\;\,\! 1}          & {\;\,\! 0}          & {\;\,\! 0}          & {\;\,\! 1}          & {\;\,\! 1}           & {\;\,\! 1}            & {\;\,\! 1}\\ 
r_4 & {\;\,\! 1}           & {\;\,\! 1}          & {\;\,\! 1}          & {\;\,\! 0}          & {\;\,\! 1}          & {\;\,\! 1}          & {\;\,\! 1}           & {\;\,\! 1}            & {\;\,\! 1}\\
r_5 & {\;\,\! 1}           & {\;\,\! 0}          & {\;\,\! 1}          & {\;\,\! 1}          & {\;\,\! 1}          & {\;\,\! 1}          & {\;\,\! 1}           & {\;\,\! 1}            & {\;\,\! 1}\\ 
r_6 & {\;\,\! 1}           & {\;\,\! 0}          & {\;\,\! 1}          & {\;\,\! 1}          & {\;\,\! 1}          & {\;\,\! 1}          & {\;\,\! 0}           & {\;\,\! 1}            & {\;\,\! 0}\\          
r_7 & {\;\,\! 1}           & {\;\,\! 0}          & {\;\,\! 1}          & {\;\,\! 1}          & {\;\,\! 1}          & {\;\,\! 0}          & {\;\,\! 0}           & {\;\,\! 0}            & {\;\,\! 0}\\
r_8 & {\;\,\! 0}           & {\;\,\! 0}          & {\;\,\! 1}          & {\;\,\! 1}          & {\;\,\! 0}          & {\;\,\! 0}          & {\;\,\! 0}           & {\;\,\! 0}            & {\;\,\! 0}
}
$
\label{fig_g13}
\end{figure*}
\noindent
\begin{center} $M(G_{13})/\{r_2,r_6,s_6,s_9\}$ contains an $M(K_{3,3})$-minor.\end{center}
{\tt \scriptsize {\bf Command:} ./macek -pGF2 '!contract 2;!contract 6;!contract -6;!contract -9;!minor' g13 '\{grK5,grK33\}'}\\
{\tt \scriptsize {\bf Output: }
The \#1 matroid [g13$\sim$c2$\sim$c6$\sim$c-6$\sim$c-9] +HAS+ minor \#2 [grK33] in the list \{grK5 grK33\}.}

\subsubsection*{The matroid $M(G_{14})$:}
\FloatBarrier
\begin{figure*}[h]
\raisebox{-10ex}{\resizebox*{0.6\width}{!}{
\psfrag{r1}{\footnotesize $r_1$}
\psfrag{r2}{\footnotesize $r_2$}
\psfrag{r3}{\footnotesize $r_3$}
\psfrag{r4}{\footnotesize $r_4$}
\psfrag{r5}{\footnotesize $r_5$}
\psfrag{r6}{\footnotesize $r_6$}
\psfrag{r7}{\footnotesize $r_7$}
\psfrag{r8}{\footnotesize $r_8$}
\psfrag{r9}{\footnotesize $r_9$}
\psfrag{s1}{\footnotesize $s_1$}
\psfrag{s2}{\footnotesize $s_2$}
\psfrag{s3}{\footnotesize $s_3$}
\psfrag{s4}{\footnotesize $s_4$}
\psfrag{s5}{\footnotesize $s_5$}
\psfrag{s6}{\footnotesize $s_6$}
\psfrag{s7}{\footnotesize $s_7$}
\psfrag{s8}{\footnotesize $s_8$}
\includegraphics [scale=0.8] {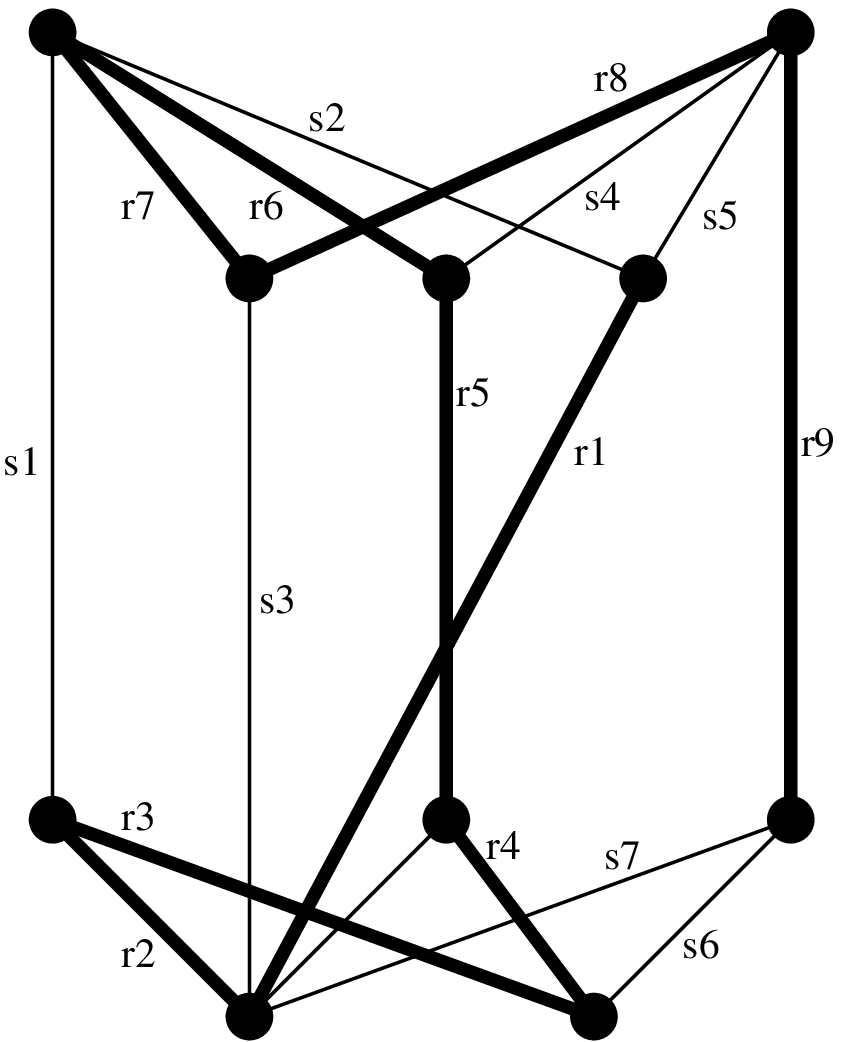}}}
$\mspace{30mu}$
$
g_{14}=
\kbordermatrix{\mbox{}&s_1&s_2&s_3&s_4&s_5&s_6&s_7&s_8\\
r_1 & {\;\,\! 0}  	   & {\;\,\! 1}          & {\;\,\! 0}          & {\;\,\! 0}          & {\;\,\! 1}          & {\;\,\! 0}          & {\;\,\! 0}            & {\;\,\! 0} \\
r_2 & {\;\,\! 0}           & {\;\,\! 1}          & {\;\,\! 1}          & {\;\,\! 0}          & {\;\,\! 1}          & {\;\,\! 0}          & {\;\,\! 1}            & {\;\,\! 1}\\ 
r_3 & {\;\,\! 1}           & {\;\,\! 1}          & {\;\,\! 1}          & {\;\,\! 0}          & {\;\,\! 1}          & {\;\,\! 0}          & {\;\,\! 1}            & {\;\,\! 1}\\ 
r_4 & {\;\,\! 1}           & {\;\,\! 1}          & {\;\,\! 1}          & {\;\,\! 0}          & {\;\,\! 1}          & {\;\,\! 1}          & {\;\,\! 1}            & {\;\,\! 1}\\
r_5 & {\;\,\! 1}           & {\;\,\! 1}          & {\;\,\! 1}          & {\;\,\! 1}          & {\;\,\! 1}          & {\;\,\! 1}          & {\;\,\! 1}            & {\;\,\! 0}\\ 
r_6 & {\;\,\! 1}           & {\;\,\! 1}          & {\;\,\! 1}          & {\;\,\! 1}          & {\;\,\! 1}          & {\;\,\! 1}          & {\;\,\! 1}            & {\;\,\! 0}\\          
r_7 & {\;\,\! 0}           & {\;\,\! 0}          & {\;\,\! 1}          & {\;\,\! 1}          & {\;\,\! 1}          & {\;\,\! 1}          & {\;\,\! 1}            & {\;\,\! 0}\\
r_8 & {\;\,\! 0}           & {\;\,\! 0}          & {\;\,\! 0}          & {\;\,\! 1}          & {\;\,\! 1}          & {\;\,\! 1}          & {\;\,\! 1}            & {\;\,\! 0}\\
r_9 & {\;\,\! 0}           & {\;\,\! 0}          & {\;\,\! 0}          & {\;\,\! 0}          & {\;\,\! 0}          & {\;\,\! 1}          & {\;\,\! 1}            & {\;\,\! 0}
}
$
\label{fig_g14}
\end{figure*}      
\noindent
\begin{center} $M(G_{14})/\{r_7,r_8,s_2,s_5\}$ contains an $M(K_{3,3})$-minor.\end{center}
{\tt \scriptsize {\bf Command:} ./macek -pGF2 '!contract 7;!contract 8;!contract -2;!contract -5;!minor' g14 '\{grK5,grK33\}'}\\
{\tt \scriptsize {\bf Output: }
The \#1 matroid [g14$\sim$c7$\sim$c8$\sim$c-2$\sim$c-5] +HAS+ minor \#2 [grK33] in the list \{grK5 grK33\}.}

\subsubsection*{The matroid $M(G_{15})$:}
\FloatBarrier
\begin{figure*}[h]
\raisebox{-10ex}{\resizebox*{0.6\width}{!}{
\psfrag{r1}{\footnotesize $r_1$}
\psfrag{r2}{\footnotesize $r_2$}
\psfrag{r3}{\footnotesize $r_3$}
\psfrag{r4}{\footnotesize $r_4$}
\psfrag{r5}{\footnotesize $r_5$}
\psfrag{r6}{\footnotesize $r_6$}
\psfrag{r7}{\footnotesize $r_7$}
\psfrag{r8}{\footnotesize $r_8$}
\psfrag{r9}{\footnotesize $r_9$}
\psfrag{s1}{\footnotesize $s_1$}
\psfrag{s2}{\footnotesize $s_2$}
\psfrag{s3}{\footnotesize $s_3$}
\psfrag{s4}{\footnotesize $s_4$}
\psfrag{s5}{\footnotesize $s_5$}
\psfrag{s6}{\footnotesize $s_6$}
\psfrag{s7}{\footnotesize $s_7$}
\psfrag{s8}{\footnotesize $s_8$}
\includegraphics [scale=0.8] {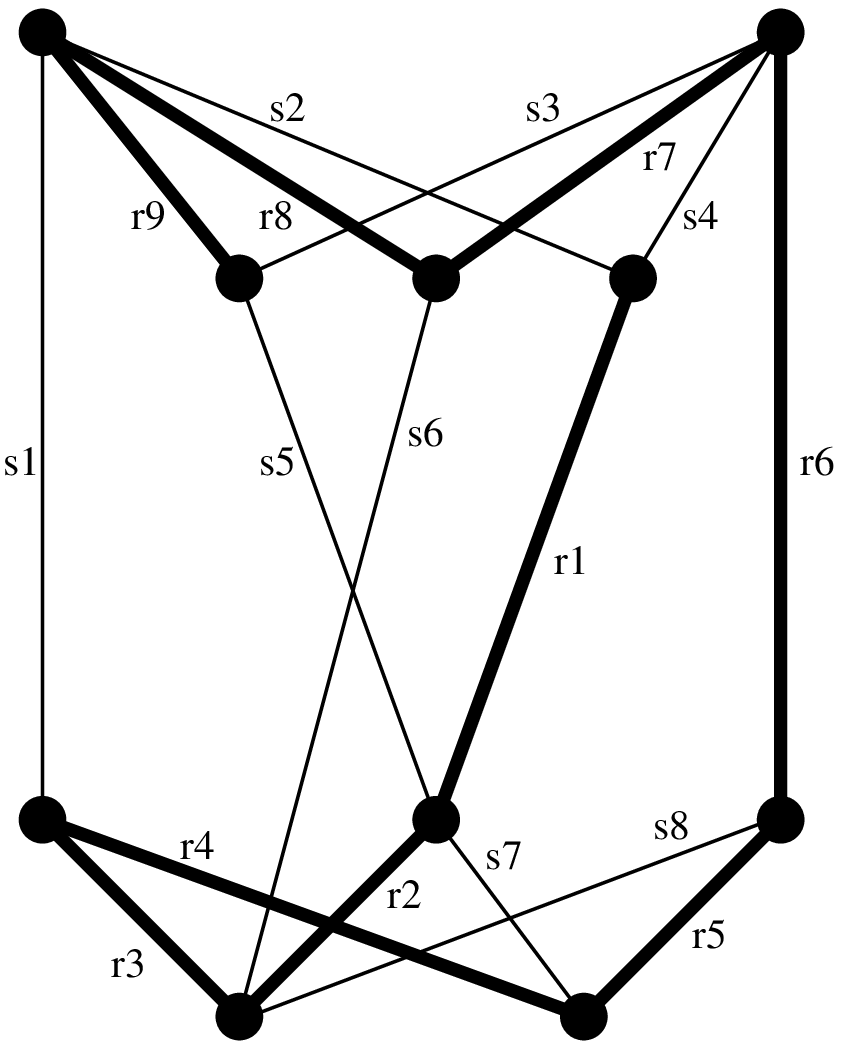}}}
$\mspace{30mu}$
$
g_{15}=
\kbordermatrix{\mbox{}&s_1&s_2&s_3&s_4&s_5&s_6&s_7&s_8\\
r_1 & {\;\,\! 0}  	   & {\;\,\! 1}          & {\;\,\! 0}          & {\;\,\! 1}          & {\;\,\! 0}          & {\;\,\! 0}          & {\;\,\! 0}            & {\;\,\! 0} \\
r_2 & {\;\,\! 0}           & {\;\,\! 1}          & {\;\,\! 0}          & {\;\,\! 1}          & {\;\,\! 1}          & {\;\,\! 0}          & {\;\,\! 1}            & {\;\,\! 0}\\ 
r_3 & {\;\,\! 0}           & {\;\,\! 1}          & {\;\,\! 0}          & {\;\,\! 1}          & {\;\,\! 1}          & {\;\,\! 1}          & {\;\,\! 1}            & {\;\,\! 1}\\ 
r_4 & {\;\,\! 1}           & {\;\,\! 1}          & {\;\,\! 0}          & {\;\,\! 1}          & {\;\,\! 1}          & {\;\,\! 1}          & {\;\,\! 1}            & {\;\,\! 1}\\
r_5 & {\;\,\! 1}           & {\;\,\! 1}          & {\;\,\! 0}          & {\;\,\! 1}          & {\;\,\! 1}          & {\;\,\! 1}          & {\;\,\! 0}            & {\;\,\! 1}\\ 
r_6 & {\;\,\! 1}           & {\;\,\! 1}          & {\;\,\! 0}          & {\;\,\! 1}          & {\;\,\! 1}          & {\;\,\! 1}          & {\;\,\! 0}            & {\;\,\! 0}\\          
r_7 & {\;\,\! 1}           & {\;\,\! 1}          & {\;\,\! 1}          & {\;\,\! 0}          & {\;\,\! 1}          & {\;\,\! 1}          & {\;\,\! 0}            & {\;\,\! 0}\\
r_8 & {\;\,\! 1}           & {\;\,\! 1}          & {\;\,\! 1}          & {\;\,\! 0}          & {\;\,\! 1}          & {\;\,\! 0}          & {\;\,\! 0}            & {\;\,\! 0}\\
r_9 & {\;\,\! 0}           & {\;\,\! 0}          & {\;\,\! 1}          & {\;\,\! 0}          & {\;\,\! 1}          & {\;\,\! 0}          & {\;\,\! 0}            & {\;\,\! 0}
}
$
\label{fig_g15}
\end{figure*}
\noindent
\begin{center} $M(G_{15})/\{r_7,r_8,r_9,s_3\}$ contains an $M(K_{3,3})$-minor.\end{center}
{\tt \scriptsize {\bf Command:} ./macek -pGF2 '!contract 7;!contract 8;!contract 9;!contract -3;!minor' g15 '\{grK5,grK33\}'}\\
{\tt \scriptsize {\bf Output: }
The \#1 matroid [g15$\sim$c7$\sim$c8$\sim$c9$\sim$c-3] +HAS+ minor \#2 [grK33] in the list \{grK5 grK33\}}.

\subsubsection*{The matroid $M(G_{16})$:}
\FloatBarrier
\begin{figure*}[h]
\raisebox{-10ex}{\resizebox*{0.6\width}{!}{
\psfrag{r1}{\footnotesize $r_1$}
\psfrag{r2}{\footnotesize $r_2$}
\psfrag{r3}{\footnotesize $r_3$}
\psfrag{r4}{\footnotesize $r_4$}
\psfrag{r5}{\footnotesize $r_5$}
\psfrag{r6}{\footnotesize $r_6$}
\psfrag{r7}{\footnotesize $r_7$}
\psfrag{r8}{\footnotesize $r_8$}
\psfrag{r9}{\footnotesize $r_9$}
\psfrag{s1}{\footnotesize $s_1$}
\psfrag{s2}{\footnotesize $s_2$}
\psfrag{s3}{\footnotesize $s_3$}
\psfrag{s4}{\footnotesize $s_4$}
\psfrag{s5}{\footnotesize $s_5$}
\psfrag{s6}{\footnotesize $s_6$}
\psfrag{s7}{\footnotesize $s_7$}
\psfrag{s8}{\footnotesize $s_8$}
\psfrag{s9}{\footnotesize $s_9$}
\includegraphics [scale=0.8] {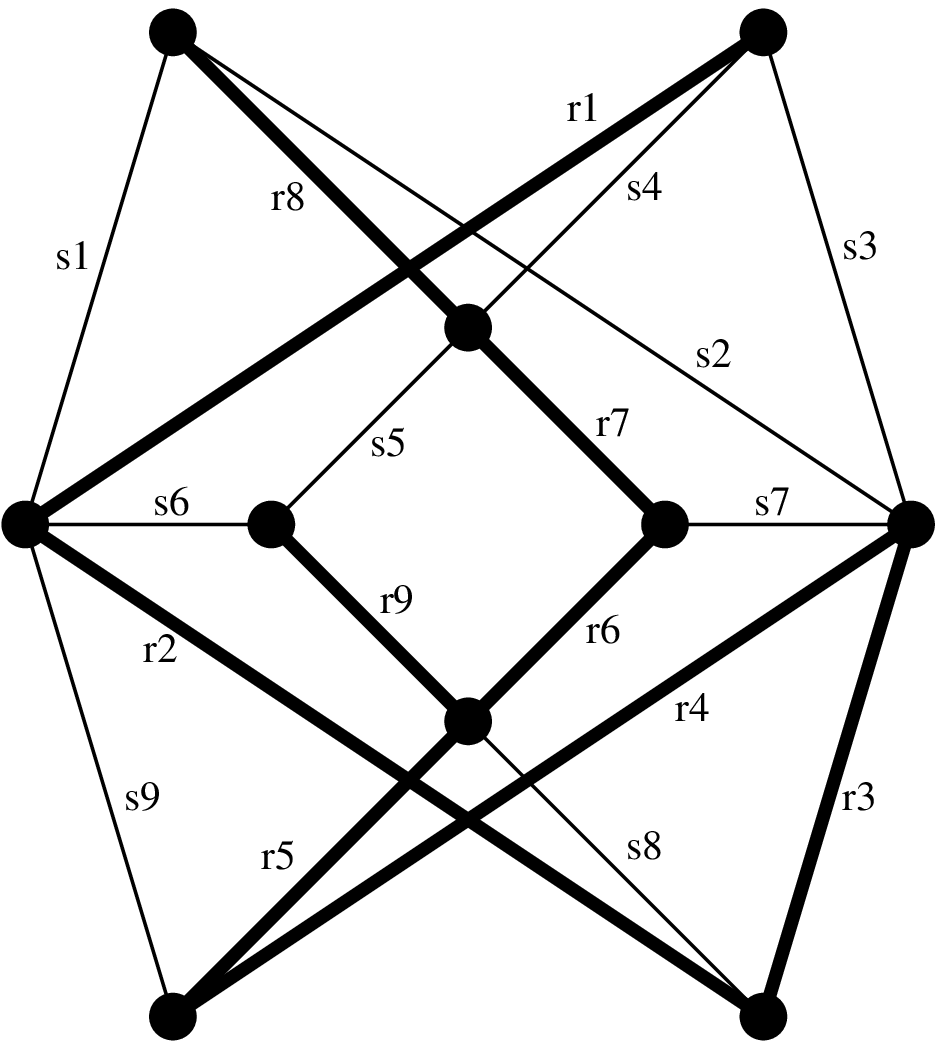}}}
$\mspace{30mu}$
$
g_{16}=
\kbordermatrix{\mbox{}&s_1&s_2&s_3&s_4&s_5&s_6&s_7&s_8&s_9\\
r_1 & {\;\,\! 0}  	   & {\;\,\! 0}          & {\;\,\! 1}          & {\;\,\! 1}          & {\;\,\! 0}          & {\;\,\! 0}          & {\;\,\! 0}            & {\;\,\! 0}        & {\;\,\! 0}\\
r_2 & {\;\,\! 1}           & {\;\,\! 0}          & {\;\,\! 1}          & {\;\,\! 1}          & {\;\,\! 0}          & {\;\,\! 1}          & {\;\,\! 0}            & {\;\,\! 0}        & {\;\,\! 1}\\ 
r_3 & {\;\,\! 1}           & {\;\,\! 0}          & {\;\,\! 1}          & {\;\,\! 1}          & {\;\,\! 0}          & {\;\,\! 1}          & {\;\,\! 0}            & {\;\,\! 1}        & {\;\,\! 1}\\ 
r_4 & {\;\,\! 1}           & {\;\,\! 1}          & {\;\,\! 0}          & {\;\,\! 1}          & {\;\,\! 0}          & {\;\,\! 1}          & {\;\,\! 1}            & {\;\,\! 1}        & {\;\,\! 1}\\
r_5 & {\;\,\! 1}           & {\;\,\! 1}          & {\;\,\! 0}          & {\;\,\! 1}          & {\;\,\! 0}          & {\;\,\! 1}          & {\;\,\! 1}            & {\;\,\! 1}        & {\;\,\! 0}\\ 
r_6 & {\;\,\! 1}           & {\;\,\! 1}          & {\;\,\! 0}          & {\;\,\! 1}          & {\;\,\! 1}          & {\;\,\! 0}          & {\;\,\! 1}            & {\;\,\! 0}        & {\;\,\! 0}\\          
r_7 & {\;\,\! 1}           & {\;\,\! 1}          & {\;\,\! 0}          & {\;\,\! 1}          & {\;\,\! 1}          & {\;\,\! 0}          & {\;\,\! 0}            & {\;\,\! 0}        & {\;\,\! 0}\\
r_8 & {\;\,\! 1}           & {\;\,\! 1}          & {\;\,\! 0}          & {\;\,\! 0}          & {\;\,\! 1}          & {\;\,\! 0}          & {\;\,\! 0}            & {\;\,\! 0}        & {\;\,\! 0}\\
r_9 & {\;\,\! 0}           & {\;\,\! 0}          & {\;\,\! 0}          & {\;\,\! 0}          & {\;\,\! 0}          & {\;\,\! 1}          & {\;\,\! 0}            & {\;\,\! 0}        & {\;\,\! 0}
}
$
\label{fig_g16}
\end{figure*}     
\noindent
\begin{center} $M(G_{16})/\{r_6,r_7,r_9,s_5\}$ contains an $M(K_{3,3})$-minor.\end{center}
{\tt \scriptsize {\bf Command:} ./macek -pGF2 '!contract 6;!contract 7;!contract 9;!contract -5;!minor' g16 '\{grK5,grK33\}'}\\
{\tt \scriptsize {\bf Output: }
The \#1 matroid [g16$\sim$c6$\sim$c7$\sim$c9$\sim$c-5] +HAS+ minor \#2 [grK33] in the list \{grK5 grK33\}.}

\subsubsection*{The matroid $M(G_{18})$:}
\FloatBarrier
\begin{figure*}[h]
\raisebox{-10ex}{\resizebox*{0.6\width}{!}{
\psfrag{r1}{\footnotesize $r_1$}
\psfrag{r2}{\footnotesize $r_2$}
\psfrag{r3}{\footnotesize $r_3$}
\psfrag{r4}{\footnotesize $r_4$}
\psfrag{r5}{\footnotesize $r_5$}
\psfrag{r6}{\footnotesize $r_6$}
\psfrag{r7}{\footnotesize $r_7$}
\psfrag{r8}{\footnotesize $r_8$}
\psfrag{r9}{\footnotesize $r_9$}
\psfrag{s1}{\footnotesize $s_1$}
\psfrag{s2}{\footnotesize $s_2$}
\psfrag{s3}{\footnotesize $s_3$}
\psfrag{s4}{\footnotesize $s_4$}
\psfrag{s5}{\footnotesize $s_5$}
\psfrag{s6}{\footnotesize $s_6$}
\includegraphics [scale=0.8] {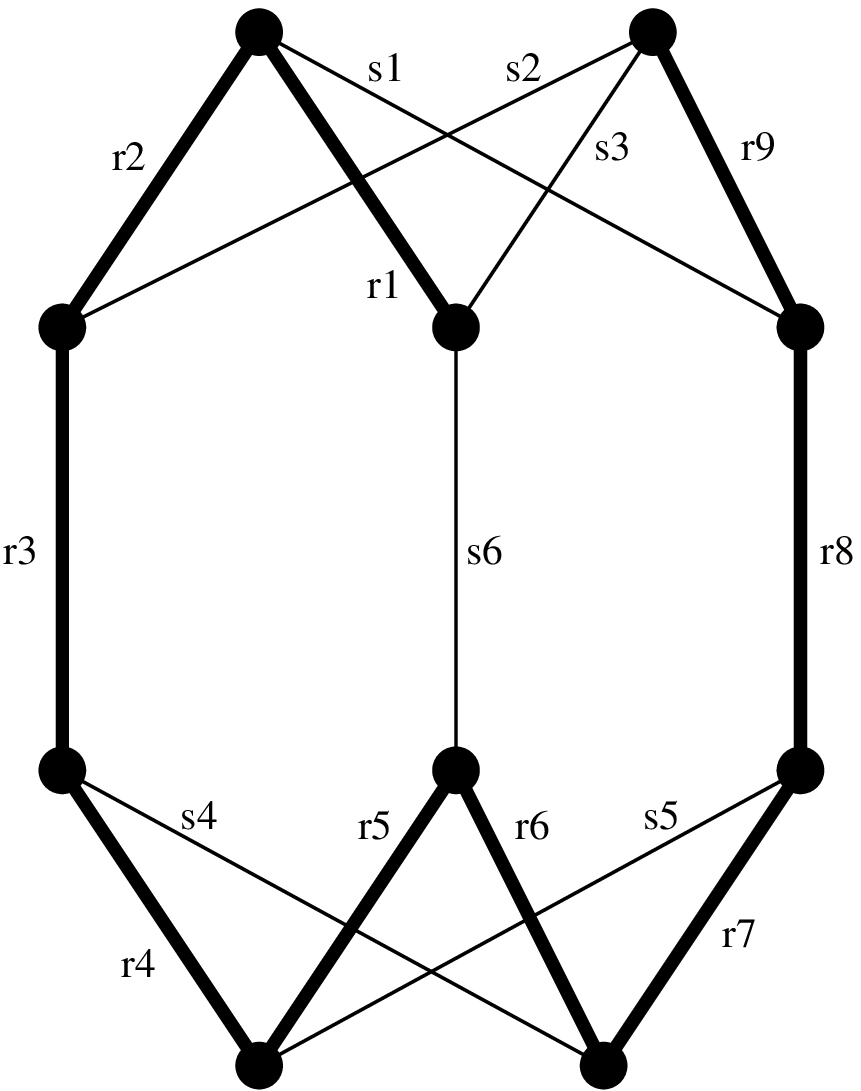}}}
$\mspace{30mu}$
$
g_{18}=
\kbordermatrix{\mbox{}&s_1&s_2&s_3&s_4&s_5&s_6\\
r_1 & {\;\,\! 0}  	   & {\;\,\! 0}          & {\;\,\! 1}          & {\;\,\! 0}          & {\;\,\! 0}          & {\;\,\! 1}         \\
r_2 & {\;\,\! 1}           & {\;\,\! 0}          & {\;\,\! 1}          & {\;\,\! 0}          & {\;\,\! 0}          & {\;\,\! 1}         \\ 
r_3 & {\;\,\! 1}           & {\;\,\! 1}          & {\;\,\! 1}          & {\;\,\! 0}          & {\;\,\! 0}          & {\;\,\! 1}         \\ 
r_4 & {\;\,\! 1}           & {\;\,\! 1}          & {\;\,\! 1}          & {\;\,\! 1}          & {\;\,\! 0}          & {\;\,\! 1}         \\
r_5 & {\;\,\! 1}           & {\;\,\! 1}          & {\;\,\! 1}          & {\;\,\! 1}          & {\;\,\! 1}          & {\;\,\! 1}         \\ 
r_6 & {\;\,\! 1}           & {\;\,\! 1}          & {\;\,\! 1}          & {\;\,\! 1}          & {\;\,\! 1}          & {\;\,\! 0}         \\          
r_7 & {\;\,\! 1}           & {\;\,\! 1}          & {\;\,\! 1}          & {\;\,\! 0}          & {\;\,\! 1}          & {\;\,\! 0}          \\
r_8 & {\;\,\! 1}           & {\;\,\! 1}          & {\;\,\! 1}          & {\;\,\! 0}          & {\;\,\! 0}          & {\;\,\! 0}          \\
r_9 & {\;\,\! 0}           & {\;\,\! 1}          & {\;\,\! 1}          & {\;\,\! 0}          & {\;\,\! 0}          & {\;\,\! 0}          
}
$
\label{fig_g19}
\end{figure*}    
\noindent
\begin{center} $M(G_{18})/\{r_4,r_7,s_4,s_5\}$ contains an $M(K_{3,3})$-minor.\end{center}
{\tt \scriptsize {\bf Command:} ./macek -pGF2 '!contract 4;!contract 7;!contract -4;!contract -5;!minor' g18 '\{grK5,grK33\}'}\\
{\tt \scriptsize {\bf Output: }
The \#1 matroid [g18$\sim$c4$\sim$c7$\sim$c-4$\sim$c-5] +HAS+ minor \#2 [grK33] in the list \{grK5 grK33\}.}

\subsubsection*{The matroid $M(G_{20})$:}
\FloatBarrier
\begin{figure*}[h]
\raisebox{-10ex}{\resizebox*{0.6\width}{!}{
\psfrag{r1}{\footnotesize $r_1$}
\psfrag{r2}{\footnotesize $r_2$}
\psfrag{r3}{\footnotesize $r_3$}
\psfrag{r4}{\footnotesize $r_4$}
\psfrag{r5}{\footnotesize $r_5$}
\psfrag{r6}{\footnotesize $r_6$}
\psfrag{r7}{\footnotesize $r_7$}
\psfrag{r8}{\footnotesize $r_8$}
\psfrag{s1}{\footnotesize $s_1$}
\psfrag{s2}{\footnotesize $s_2$}
\psfrag{s3}{\footnotesize $s_3$}
\psfrag{s4}{\footnotesize $s_4$}
\psfrag{s5}{\footnotesize $s_5$}
\psfrag{s6}{\footnotesize $s_6$}
\psfrag{s7}{\footnotesize $s_7$}
\psfrag{s8}{\footnotesize $s_8$}
\includegraphics [scale=0.8] {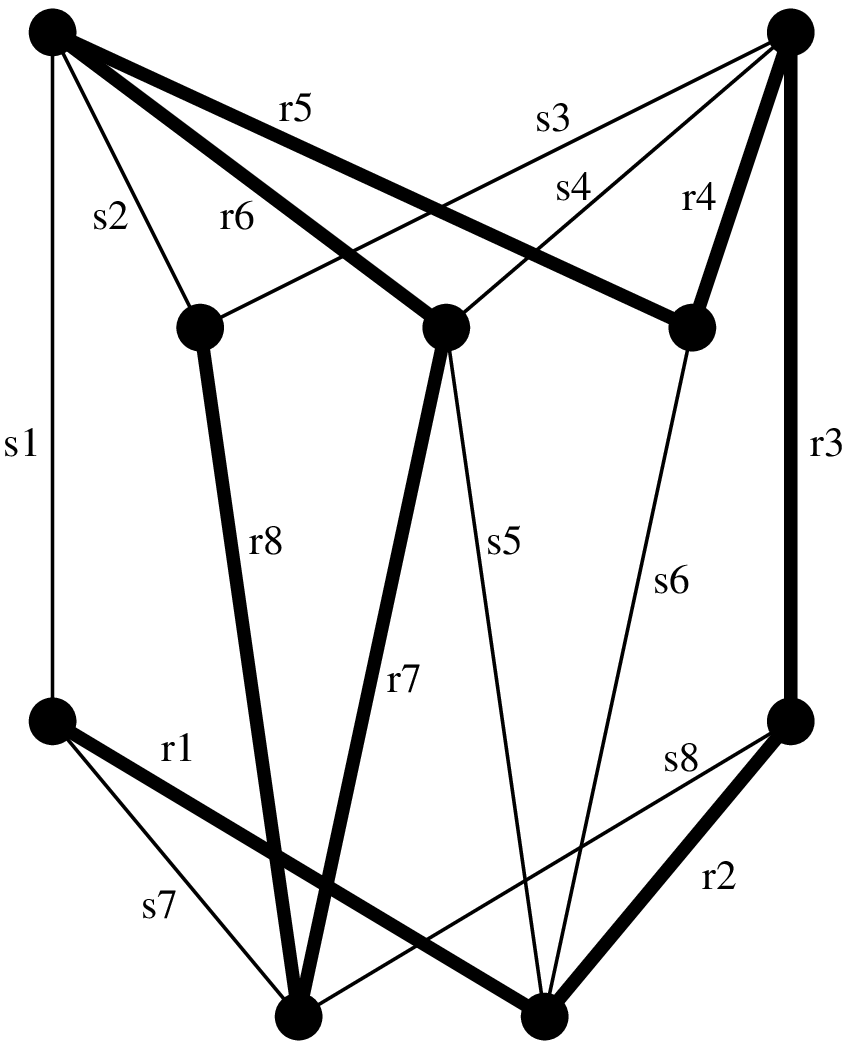}}}
$\mspace{30mu}$
$
g_{20}=
\kbordermatrix{\mbox{}&s_1&s_2&s_3&s_4&s_5&s_6&s_7&s_8\\
r_1 & {\;\,\! 1}  	   & {\;\,\! 0}          & {\;\,\! 0}          & {\;\,\! 0}          & {\;\,\! 0}          & {\;\,\! 0}         & {\;\,\! 1}        & {\;\,\! 0}\\
r_2 & {\;\,\! 1}           & {\;\,\! 0}          & {\;\,\! 0}          & {\;\,\! 0}          & {\;\,\! 1}          & {\;\,\! 1}         & {\;\,\! 1}        & {\;\,\! 0}\\ 
r_3 & {\;\,\! 1}           & {\;\,\! 0}          & {\;\,\! 0}          & {\;\,\! 0}          & {\;\,\! 1}          & {\;\,\! 1}         & {\;\,\! 1}        & {\;\,\! 1}\\ 
r_4 & {\;\,\! 1}           & {\;\,\! 0}          & {\;\,\! 1}          & {\;\,\! 1}          & {\;\,\! 1}          & {\;\,\! 1}         & {\;\,\! 1}        & {\;\,\! 1}\\
r_5 & {\;\,\! 1}           & {\;\,\! 0}          & {\;\,\! 1}          & {\;\,\! 1}          & {\;\,\! 1}          & {\;\,\! 0}         & {\;\,\! 1}        & {\;\,\! 1}\\ 
r_6 & {\;\,\! 0}           & {\;\,\! 1}          & {\;\,\! 1}          & {\;\,\! 1}          & {\;\,\! 1}          & {\;\,\! 0}         & {\;\,\! 1}        & {\;\,\! 1}\\          
r_7 & {\;\,\! 0}           & {\;\,\! 1}          & {\;\,\! 1}          & {\;\,\! 0}          & {\;\,\! 0}          & {\;\,\! 0}         & {\;\,\! 1}        & {\;\,\! 1} \\
r_8 & {\;\,\! 0}           & {\;\,\! 1}          & {\;\,\! 1}          & {\;\,\! 0}          & {\;\,\! 0}          & {\;\,\! 0}         & {\;\,\! 0}        & {\;\,\! 0}
}
$
\label{fig_g20}
\end{figure*}    
\noindent
\begin{center} $M(G_{20})/\{r_1,r_2,s_7,s_8\}$ contains an $M(K_{3,3})$-minor.\end{center}
{\tt \scriptsize {\bf Command:} ./macek -pGF2 '!contract 1;!contract 2;!contract -7;!contract -8;!minor' g20 '\{grK5,grK33\}'}\\
{\tt \scriptsize {\bf Output: }
The \#1 matroid [g20$\sim$c1$\sim$c2$\sim$c-7$\sim$c-8] +HAS+ minor \#2 [grK33] in the list \{grK5 grK33\}.}

\subsubsection*{The matroid $M(G_{21})$:}
\FloatBarrier
\begin{figure*}[h]
\raisebox{-10ex}{\resizebox*{0.6\width}{!}{
\psfrag{r1}{\footnotesize $r_1$}
\psfrag{r2}{\footnotesize $r_2$}
\psfrag{r3}{\footnotesize $r_3$}
\psfrag{r4}{\footnotesize $r_4$}
\psfrag{r5}{\footnotesize $r_5$}
\psfrag{r6}{\footnotesize $r_6$}
\psfrag{r7}{\footnotesize $r_7$}
\psfrag{s1}{\footnotesize $s_1$}
\psfrag{s2}{\footnotesize $s_2$}
\psfrag{s3}{\footnotesize $s_3$}
\psfrag{s4}{\footnotesize $s_4$}
\psfrag{s5}{\footnotesize $s_5$}
\psfrag{s6}{\footnotesize $s_6$}
\psfrag{s7}{\footnotesize $s_7$}
\psfrag{s8}{\footnotesize $s_8$}
\psfrag{s9}{\footnotesize $s_9$}
\includegraphics [scale=0.8] {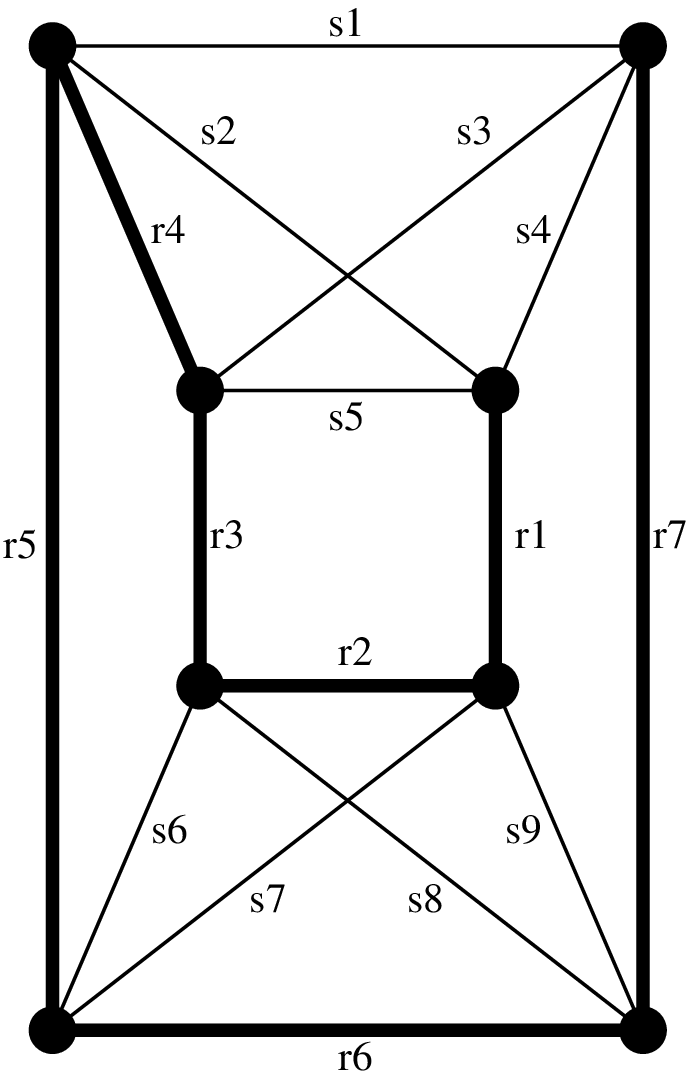}}}
$\mspace{30mu}$
$
g_{21}=
\kbordermatrix{\mbox{}&s_1&s_2&s_3&s_4&s_5&s_6&s_7&s_8&s_9\\
r_1 & {\;\,\! 0}  	   & {\;\,\! 1}          & {\;\,\! 0}          & {\;\,\! 1}          & {\;\,\! 1}          & {\;\,\! 0}         & {\;\,\! 0}        & {\;\,\! 0}       & {\;\,\! 0}\\
r_2 & {\;\,\! 0}           & {\;\,\! 1}          & {\;\,\! 0}          & {\;\,\! 1}          & {\;\,\! 1}          & {\;\,\! 0}         & {\;\,\! 1}        & {\;\,\! 0}       & {\;\,\! 1}\\ 
r_3 & {\;\,\! 0}           & {\;\,\! 1}          & {\;\,\! 0}          & {\;\,\! 1}          & {\;\,\! 1}          & {\;\,\! 1}         & {\;\,\! 1}        & {\;\,\! 1}       & {\;\,\! 1}\\ 
r_4 & {\;\,\! 0}           & {\;\,\! 1}          & {\;\,\! 1}          & {\;\,\! 1}          & {\;\,\! 0}          & {\;\,\! 1}         & {\;\,\! 1}        & {\;\,\! 1}       & {\;\,\! 1}\\
r_5 & {\;\,\! 1}           & {\;\,\! 0}          & {\;\,\! 1}          & {\;\,\! 1}          & {\;\,\! 0}          & {\;\,\! 1}         & {\;\,\! 1}        & {\;\,\! 1}       & {\;\,\! 1}\\ 
r_6 & {\;\,\! 1}           & {\;\,\! 0}          & {\;\,\! 1}          & {\;\,\! 1}          & {\;\,\! 0}          & {\;\,\! 0}         & {\;\,\! 0}        & {\;\,\! 1}       & {\;\,\! 1}\\          
r_7 & {\;\,\! 1}           & {\;\,\! 0}          & {\;\,\! 1}          & {\;\,\! 1}          & {\;\,\! 0}          & {\;\,\! 0}         & {\;\,\! 0}        & {\;\,\! 0}       & {\;\,\! 0}
}
$
\label{fig_g21}
\end{figure*}
\noindent
\begin{center} $M(G_{21})/\{r_4,s_2,s_5\}$ contains an $M(K_{5})$-minor.\end{center}
{\tt \scriptsize {\bf Command:} ./macek -pGF2 '!contract 4;!contract -2;!contract -5;!minor' g21 '\{grK5,grK33\}'}\\
{\tt \scriptsize {\bf Output: }
The \#1 matroid [g21$\sim$c4$\sim$c-2$\sim$c-5] +HAS+ minor \#1 [grK5] in the list \{grK5 grK33\}.}

\subsubsection*{The matroid $M(G_{22})$:}
\FloatBarrier
\begin{figure*}[h]
\raisebox{-10ex}{\resizebox*{0.6\width}{!}{
\psfrag{r1}{\footnotesize $r_1$}
\psfrag{r2}{\footnotesize $r_2$}
\psfrag{r3}{\footnotesize $r_3$}
\psfrag{r4}{\footnotesize $r_4$}
\psfrag{r5}{\footnotesize $r_5$}
\psfrag{r6}{\footnotesize $r_6$}
\psfrag{r7}{\footnotesize $r_7$}
\psfrag{r8}{\footnotesize $r_8$}
\psfrag{s1}{\footnotesize $s_1$}
\psfrag{s2}{\footnotesize $s_2$}
\psfrag{s3}{\footnotesize $s_3$}
\psfrag{s4}{\footnotesize $s_4$}
\psfrag{s5}{\footnotesize $s_5$}
\psfrag{s6}{\footnotesize $s_6$}
\psfrag{s7}{\footnotesize $s_7$}
\psfrag{s8}{\footnotesize $s_8$}
\includegraphics [scale=0.8] {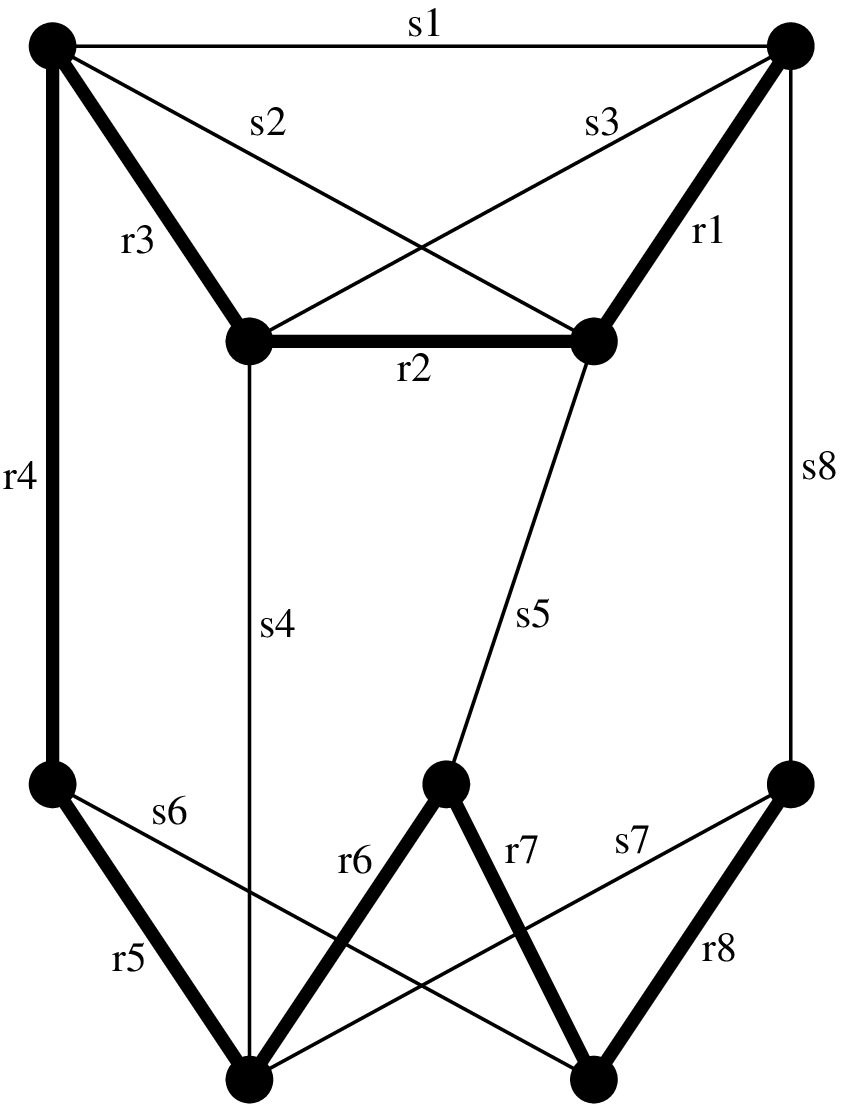}}}
$\mspace{30mu}$
$
g_{22}=
\kbordermatrix{\mbox{}&s_1&s_2&s_3&s_4&s_5&s_6&s_7&s_8\\
r_1 & {\;\,\! 1}  	   & {\;\,\! 0}          & {\;\,\! 1}          & {\;\,\! 0}          & {\;\,\! 0}          & {\;\,\! 0}         & {\;\,\! 0}        & {\;\,\! 1}      \\
r_2 & {\;\,\! 1}           & {\;\,\! 1}          & {\;\,\! 1}          & {\;\,\! 0}          & {\;\,\! 1}          & {\;\,\! 0}         & {\;\,\! 0}        & {\;\,\! 1}      \\ 
r_3 & {\;\,\! 1}           & {\;\,\! 1}          & {\;\,\! 0}          & {\;\,\! 1}          & {\;\,\! 1}          & {\;\,\! 0}         & {\;\,\! 0}        & {\;\,\! 1}      \\ 
r_4 & {\;\,\! 0}           & {\;\,\! 0}          & {\;\,\! 0}          & {\;\,\! 1}          & {\;\,\! 1}          & {\;\,\! 0}         & {\;\,\! 0}        & {\;\,\! 1}      \\
r_5 & {\;\,\! 0}           & {\;\,\! 0}          & {\;\,\! 0}          & {\;\,\! 1}          & {\;\,\! 1}          & {\;\,\! 1}         & {\;\,\! 0}        & {\;\,\! 1}      \\ 
r_6 & {\;\,\! 0}           & {\;\,\! 0}          & {\;\,\! 0}          & {\;\,\! 0}          & {\;\,\! 1}          & {\;\,\! 1}         & {\;\,\! 1}        & {\;\,\! 1}      \\          
r_7 & {\;\,\! 0}           & {\;\,\! 0}          & {\;\,\! 0}          & {\;\,\! 0}          & {\;\,\! 0}          & {\;\,\! 1}         & {\;\,\! 1}        & {\;\,\! 1}      \\
r_8 & {\;\,\! 0}           & {\;\,\! 0}          & {\;\,\! 0}          & {\;\,\! 0}          & {\;\,\! 0}          & {\;\,\! 0}         & {\;\,\! 1}        & {\;\,\! 1}
}
$
\label{fig_g22}
\end{figure*}     
\noindent
\begin{center} $M(G_{22})/\{r_5,r_6,r_7,s_6\}$ contains an $M(K_{5})$-minor.\end{center}
{\tt \scriptsize {\bf Command:} ./macek -pGF2 '!contract 5;!contract 6;!contract 7;!contract -6;!minor' g22 '\{grK5,grK33\}'}\\
{\tt \scriptsize {\bf Output: }
 The \#1 matroid [g22$\sim$c5$\sim$c6$\sim$c7$\sim$c-6] +HAS+ minor \#1 [grK5] in the list \{grK5 grK33\}.}

\subsubsection*{The matroid $M(G_{23})$:}
\FloatBarrier
\begin{figure*}[h]
\raisebox{-10ex}{\resizebox*{0.6\width}{!}{
\psfrag{r1}{\footnotesize $r_1$}
\psfrag{r2}{\footnotesize $r_2$}
\psfrag{r3}{\footnotesize $r_3$}
\psfrag{r4}{\footnotesize $r_4$}
\psfrag{r5}{\footnotesize $r_5$}
\psfrag{r6}{\footnotesize $r_6$}
\psfrag{r7}{\footnotesize $r_7$}
\psfrag{r8}{\footnotesize $r_8$}
\psfrag{r9}{\footnotesize $r_9$}
\psfrag{s1}{\footnotesize $s_1$}
\psfrag{s2}{\footnotesize $s_2$}
\psfrag{s3}{\footnotesize $s_3$}
\psfrag{s4}{\footnotesize $s_4$}
\psfrag{s5}{\footnotesize $s_5$}
\psfrag{s6}{\footnotesize $s_6$}
\psfrag{s7}{\footnotesize $s_7$}
\includegraphics [scale=0.8] {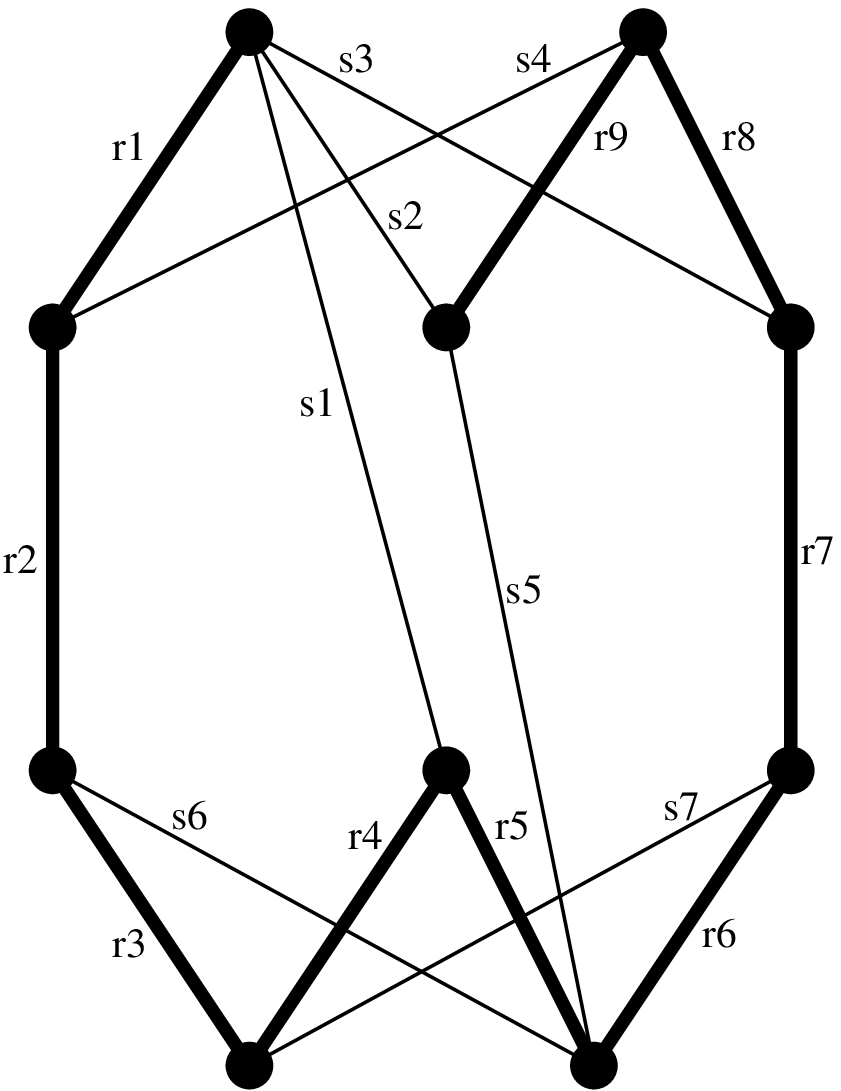}}}
$\mspace{30mu}$
$
g_{23}=
\kbordermatrix{\mbox{}&s_1&s_2&s_3&s_4&s_5&s_6&s_7\\
r_1 & {\;\,\! 1}  	   & {\;\,\! 1}          & {\;\,\! 1}          & {\;\,\! 0}          & {\;\,\! 0}          & {\;\,\! 0}         & {\;\,\! 0}            \\
r_2 & {\;\,\! 1}           & {\;\,\! 1}          & {\;\,\! 1}          & {\;\,\! 1}          & {\;\,\! 0}          & {\;\,\! 0}         & {\;\,\! 0}            \\ 
r_3 & {\;\,\! 1}           & {\;\,\! 1}          & {\;\,\! 1}          & {\;\,\! 1}          & {\;\,\! 0}          & {\;\,\! 1}         & {\;\,\! 0}            \\ 
r_4 & {\;\,\! 1}           & {\;\,\! 1}          & {\;\,\! 1}          & {\;\,\! 1}          & {\;\,\! 0}          & {\;\,\! 1}         & {\;\,\! 1}            \\
r_5 & {\;\,\! 0}           & {\;\,\! 1}          & {\;\,\! 1}          & {\;\,\! 1}          & {\;\,\! 0}          & {\;\,\! 1}         & {\;\,\! 1}             \\ 
r_6 & {\;\,\! 0}           & {\;\,\! 1}          & {\;\,\! 1}          & {\;\,\! 1}          & {\;\,\! 1}          & {\;\,\! 0}         & {\;\,\! 1}            \\          
r_7 & {\;\,\! 0}           & {\;\,\! 1}          & {\;\,\! 1}          & {\;\,\! 1}          & {\;\,\! 1}          & {\;\,\! 0}         & {\;\,\! 0}             \\
r_8 & {\;\,\! 0}           & {\;\,\! 1}          & {\;\,\! 0}          & {\;\,\! 1}          & {\;\,\! 1}          & {\;\,\! 0}         & {\;\,\! 0}            \\
r_9 & {\;\,\! 0}           & {\;\,\! 1}          & {\;\,\! 0}          & {\;\,\! 0}          & {\;\,\! 1}          & {\;\,\! 0}         & {\;\,\! 0}       
}
$
\label{fig_g23}
\end{figure*}      
\noindent
\begin{center} $M(G_{23})/\{r_1,r_8,s_3,s_4\}$ contains an $M(K_{3,3})$-minor.\end{center}
{\tt \scriptsize {\bf Command:} ./macek -pGF2 '!contract 1;!contract 8;!contract -3;!contract -4;!minor' g23 '\{grK5,grK33\}'}\\
{\tt \scriptsize {\bf Output: } The \#1 matroid [g23$\sim$c1$\sim$c8$\sim$c-3$\sim$c-4] +HAS+ minor \#2 [grK33] in the list \{grK5 grK33\}.} \\

\subsubsection*{The matroid $M(G_{24})$:}
\FloatBarrier
\begin{figure*}[h]
\raisebox{-10ex}{\resizebox*{0.6\width}{!}{
\psfrag{r1}{\footnotesize $r_1$}
\psfrag{r2}{\footnotesize $r_2$}
\psfrag{r3}{\footnotesize $r_3$}
\psfrag{r4}{\footnotesize $r_4$}
\psfrag{r5}{\footnotesize $r_5$}
\psfrag{r6}{\footnotesize $r_6$}
\psfrag{s1}{\footnotesize $s_1$}
\psfrag{s2}{\footnotesize $s_2$}
\psfrag{s3}{\footnotesize $s_3$}
\psfrag{s4}{\footnotesize $s_4$}
\psfrag{s5}{\footnotesize $s_5$}
\psfrag{s6}{\footnotesize $s_6$}
\psfrag{s7}{\footnotesize $s_7$}
\psfrag{s8}{\footnotesize $s_8$}
\psfrag{s9}{\footnotesize $s_9$}
\psfrag{s10}{\footnotesize $s_{10}$}
\psfrag{s11}{\footnotesize $s_{11}$}
\psfrag{s12}{\footnotesize $s_{12}$}
\includegraphics [scale=0.8] {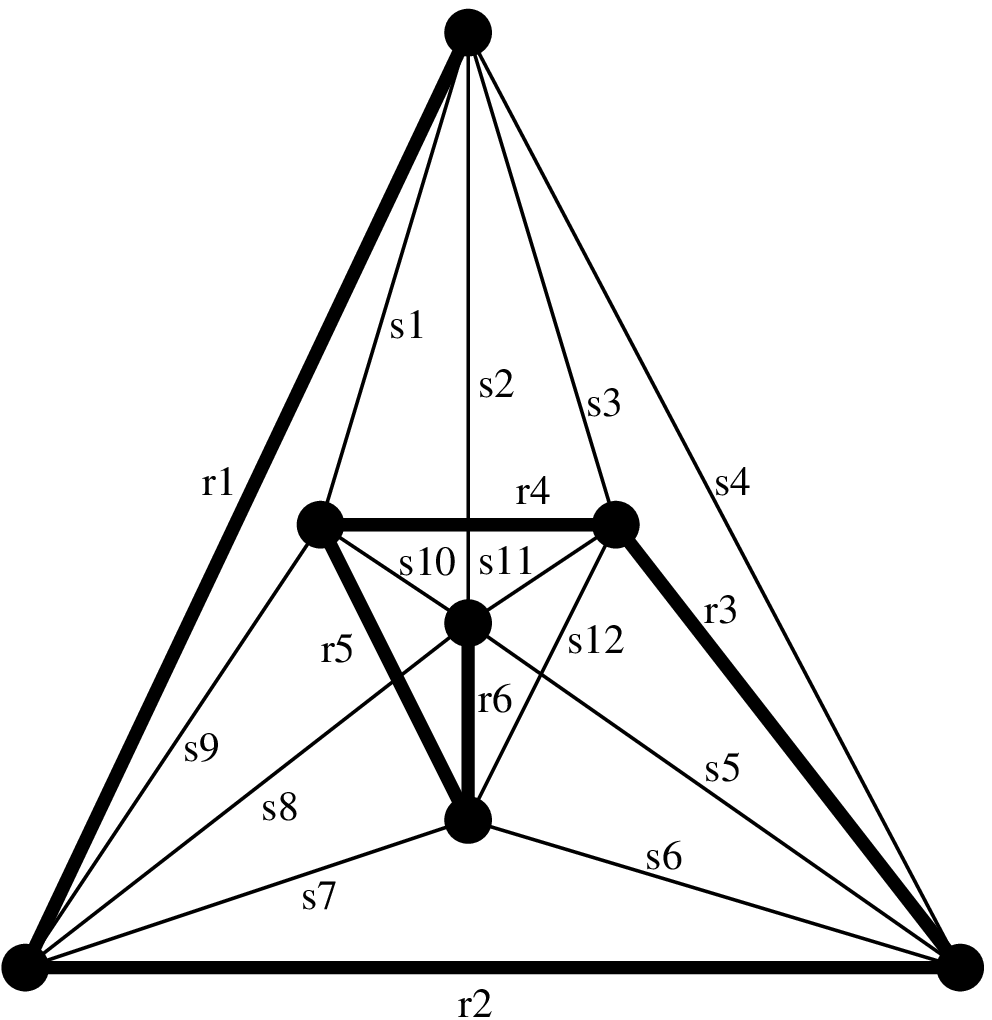}}}
$\mspace{30mu}$
$
g_{24}=
\kbordermatrix{\mbox{}&s_1&s_2&s_3&s_4&s_5&s_6&s_7&s8&s9&s_{10}&s_{11}&s_{12}\\
r_1 & {\;\,\! 1}  	   & {\;\,\! 1}          & {\;\,\! 1}          & {\;\,\! 1}          & {\;\,\! 0}          & {\;\,\! 0}          & {\;\,\! 0}           & {\;\,\! 0}       & {\;\,\! 0}      & {\;\,\! 0}       & {\;\,\! 0}    & {\;\,\! 0}\\
r_2 & {\;\,\! 1}           & {\;\,\! 1}          & {\;\,\! 1}          & {\;\,\! 1}          & {\;\,\! 0}          & {\;\,\! 0}          & {\;\,\! 1}           & {\;\,\! 1}       & {\;\,\! 1}      & {\;\,\! 0}       & {\;\,\! 0}    & {\;\,\! 0}\\ 
r_3 & {\;\,\! 1}           & {\;\,\! 1}          & {\;\,\! 1}          & {\;\,\! 0}          & {\;\,\! 1}          & {\;\,\! 1}          & {\;\,\! 1}           & {\;\,\! 1}       & {\;\,\! 1}      & {\;\,\! 0}       & {\;\,\! 0}    & {\;\,\! 0}\\ 
r_4 & {\;\,\! 1}           & {\;\,\! 1}          & {\;\,\! 0}          & {\;\,\! 0}          & {\;\,\! 1}          & {\;\,\! 1}          & {\;\,\! 1}           & {\;\,\! 1}       & {\;\,\! 1}      & {\;\,\! 0}       & {\;\,\! 1}    & {\;\,\! 1}\\
r_5 & {\;\,\! 0}           & {\;\,\! 1}          & {\;\,\! 0}          & {\;\,\! 0}          & {\;\,\! 1}          & {\;\,\! 1}          & {\;\,\! 1}           & {\;\,\! 1}       & {\;\,\! 0}      & {\;\,\! 1}       & {\;\,\! 1}    & {\;\,\! 1}\\ 
r_6 & {\;\,\! 0}           & {\;\,\! 1}          & {\;\,\! 0}          & {\;\,\! 0}          & {\;\,\! 1}          & {\;\,\! 0}          & {\;\,\! 0}           & {\;\,\! 1}       & {\;\,\! 0}      & {\;\,\! 1}       & {\;\,\! 1}    & {\;\,\! 0}      
}   
$
\label{fig_g24}
\end{figure*}
\noindent
\begin{center} $M(G_{24})/\{r_1,s_1,s_9\}$ contains an $M(K_5)$-minor.\end{center}
{\tt \scriptsize {\bf Command:} ./macek -pGF2 '!contract 1;!contract -1;!contract -9;!minor' g24 '\{grK5,grK33\}'}\\
{\tt \scriptsize {\bf Output: }The \#1 matroid [g24$\sim$c1$\sim$c-1$\sim$c-9] +HAS+ minor \#1 [grK5] in the list \{grK5 grK33\}.}

\subsubsection*{The matroid $M(G_{25})$:}
\FloatBarrier
\begin{figure*}[h]
\raisebox{-10ex}{\resizebox*{0.6\width}{!}{
\psfrag{r1}{\footnotesize $r_1$}
\psfrag{r2}{\footnotesize $r_2$}
\psfrag{r3}{\footnotesize $r_3$}
\psfrag{r4}{\footnotesize $r_4$}
\psfrag{r5}{\footnotesize $r_5$}
\psfrag{r6}{\footnotesize $r_6$}
\psfrag{r7}{\footnotesize $r_7$}
\psfrag{s1}{\footnotesize $s_1$}
\psfrag{s2}{\footnotesize $s_2$}
\psfrag{s3}{\footnotesize $s_3$}
\psfrag{s4}{\footnotesize $s_4$}
\psfrag{s5}{\footnotesize $s_5$}
\psfrag{s6}{\footnotesize $s_6$}
\psfrag{s7}{\footnotesize $s_7$}
\psfrag{s8}{\footnotesize $s_8$}
\psfrag{s9}{\footnotesize $s_9$}
\psfrag{s10}{\footnotesize $s_{10}$}
\psfrag{s11}{\footnotesize $s_{11}$}
\psfrag{s12}{\footnotesize $s_{12}$}
\psfrag{s13}{\footnotesize $s_{13}$}
\includegraphics [scale=0.8] {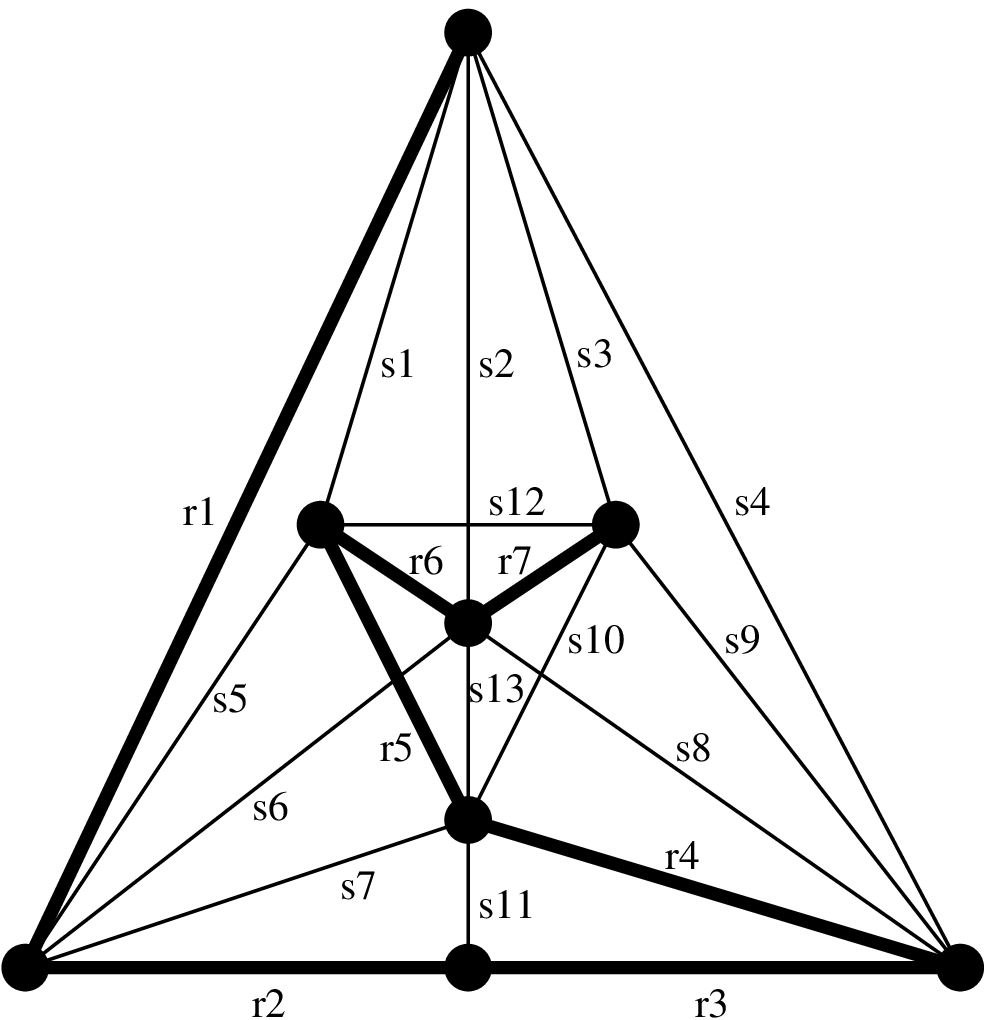}}}
$\mspace{30mu}$
$
g_{25}=
\kbordermatrix{\mbox{}&s_1&s_2&s_3&s_4&s_5&s_6&s_7&s8&s9&s_{10}&s_{11}&s_{12}&s_{13}\\
r_1 & {\;\,\! 1}  	   & {\;\,\! 1}          & {\;\,\! 1}          & {\;\,\! 1}          & {\;\,\! 0}          & {\;\,\! 0}          & {\;\,\! 0}           & {\;\,\! 0}       & {\;\,\! 0}      & {\;\,\! 0}       & {\;\,\! 0}    & {\;\,\! 0}     & {\;\,\! 0}           \\
r_2 & {\;\,\! 1}           & {\;\,\! 1}          & {\;\,\! 1}          & {\;\,\! 1}          & {\;\,\! 1}          & {\;\,\! 1}          & {\;\,\! 1}           & {\;\,\! 0}       & {\;\,\! 0}      & {\;\,\! 0}       & {\;\,\! 0}    & {\;\,\! 0}     & {\;\,\! 0}\\ 
r_3 & {\;\,\! 1}           & {\;\,\! 1}          & {\;\,\! 1}          & {\;\,\! 1}          & {\;\,\! 1}          & {\;\,\! 1}          & {\;\,\! 1}           & {\;\,\! 0}       & {\;\,\! 0}      & {\;\,\! 0}       & {\;\,\! 1}    & {\;\,\! 0}     & {\;\,\! 0}\\ 
r_4 & {\;\,\! 1}           & {\;\,\! 1}          & {\;\,\! 1}          & {\;\,\! 0}          & {\;\,\! 1}          & {\;\,\! 1}          & {\;\,\! 1}           & {\;\,\! 1}       & {\;\,\! 1}      & {\;\,\! 0}       & {\;\,\! 1}    & {\;\,\! 0}     & {\;\,\! 0}\\
r_5 & {\;\,\! 1}           & {\;\,\! 1}          & {\;\,\! 1}          & {\;\,\! 0}          & {\;\,\! 1}          & {\;\,\! 1}          & {\;\,\! 0}           & {\;\,\! 1}       & {\;\,\! 1}      & {\;\,\! 1}       & {\;\,\! 0}    & {\;\,\! 0}     & {\;\,\! 1}\\ 
r_6 & {\;\,\! 0}           & {\;\,\! 1}          & {\;\,\! 1}          & {\;\,\! 0}          & {\;\,\! 0}          & {\;\,\! 1}          & {\;\,\! 0}           & {\;\,\! 1}       & {\;\,\! 1}      & {\;\,\! 1}       & {\;\,\! 0}    & {\;\,\! 1}     & {\;\,\! 1}\\
r_7 & {\;\,\! 0}           & {\;\,\! 0}          & {\;\,\! 1}          & {\;\,\! 0}          & {\;\,\! 0}          & {\;\,\! 0}          & {\;\,\! 0}           & {\;\,\! 0}       & {\;\,\! 1}      & {\;\,\! 1}       & {\;\,\! 0}    & {\;\,\! 1}     & {\;\,\! 0}   
}   
$
\label{fig_g25}
\end{figure*}
\noindent
\begin{center} $M(G_{25})/\{r_5,s_{10},s_{12}\}$ contains an $M(K_5)$-minor.\end{center}
{\tt \scriptsize {\bf Command:} ./macek -pGF2 '!contract 5;!contract -10;!contract -12;!minor' g25 '\{grK5,grK33\}'}\\
{\tt \scriptsize {\bf Output: }The \#1 matroid [g25$\sim$c5$\sim$c-10$\sim$c-12] +HAS+ minor \#1 [grK5] in the list \{grK5 grK33\}.}

\subsubsection*{The matroid $M(G_{26})$:}
\FloatBarrier
\begin{figure*}[h]
\raisebox{-10ex}{\resizebox*{0.6\width}{!}{
\psfrag{r1}{\footnotesize $r_1$}
\psfrag{r2}{\footnotesize $r_2$}
\psfrag{r3}{\footnotesize $r_3$}
\psfrag{r4}{\footnotesize $r_4$}
\psfrag{r5}{\footnotesize $r_5$}
\psfrag{r6}{\footnotesize $r_6$}
\psfrag{r7}{\footnotesize $r_7$}
\psfrag{r8}{\footnotesize $r_8$}
\psfrag{s1}{\footnotesize $s_1$}
\psfrag{s2}{\footnotesize $s_2$}
\psfrag{s3}{\footnotesize $s_3$}
\psfrag{s4}{\footnotesize $s_4$}
\psfrag{s5}{\footnotesize $s_5$}
\psfrag{s6}{\footnotesize $s_6$}
\psfrag{s7}{\footnotesize $s_7$}
\psfrag{s8}{\footnotesize $s_8$}
\psfrag{s9}{\footnotesize $s_9$}
\psfrag{s10}{\footnotesize $s_{10}$}
\psfrag{s11}{\footnotesize $s_{11}$}
\includegraphics [scale=0.8] {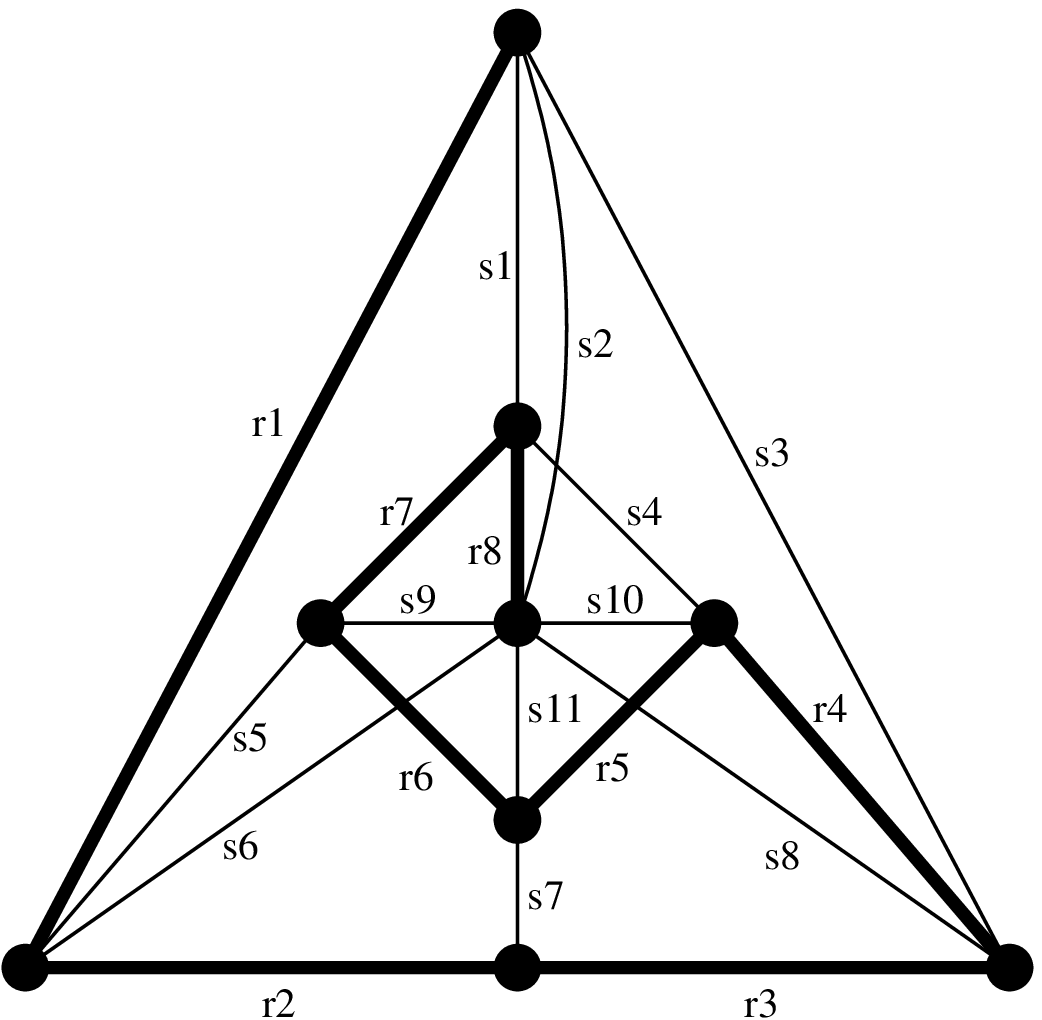}}}
$\mspace{30mu}$
$
g_{26}=
\kbordermatrix{\mbox{}&s_1&s_2&s_3&s_4&s_5&s_6&s_7&s_8&s_9&s_{10}&s_{11}\\
r_1 & {\;\,\! 1}  	   & {\;\,\! 1}          & {\;\,\! 1}          & {\;\,\! 0}          & {\;\,\! 0}          & {\;\,\! 0}          & {\;\,\! 0}           & {\;\,\! 0}       & {\;\,\! 0}      & {\;\,\! 0}             & {\;\,\! 0}\\
r_2 & {\;\,\! 1}           & {\;\,\! 1}          & {\;\,\! 1}          & {\;\,\! 0}          & {\;\,\! 1}          & {\;\,\! 1}          & {\;\,\! 0}           & {\;\,\! 0}       & {\;\,\! 0}      & {\;\,\! 0}             & {\;\,\! 0}\\ 
r_3 & {\;\,\! 1}           & {\;\,\! 1}          & {\;\,\! 1}          & {\;\,\! 0}          & {\;\,\! 1}          & {\;\,\! 1}          & {\;\,\! 1}           & {\;\,\! 0}       & {\;\,\! 0}      & {\;\,\! 0}             & {\;\,\! 0}\\ 
r_4 & {\;\,\! 1}           & {\;\,\! 1}          & {\;\,\! 0}          & {\;\,\! 0}          & {\;\,\! 1}          & {\;\,\! 1}          & {\;\,\! 1}           & {\;\,\! 1}       & {\;\,\! 0}      & {\;\,\! 0}             & {\;\,\! 0}\\
r_5 & {\;\,\! 1}           & {\;\,\! 1}          & {\;\,\! 0}          & {\;\,\! 1}          & {\;\,\! 1}          & {\;\,\! 1}          & {\;\,\! 1}           & {\;\,\! 1}       & {\;\,\! 0}      & {\;\,\! 1}             & {\;\,\! 0}\\ 
r_6 & {\;\,\! 1}           & {\;\,\! 1}          & {\;\,\! 0}          & {\;\,\! 1}          & {\;\,\! 1}          & {\;\,\! 1}          & {\;\,\! 0}           & {\;\,\! 1}       & {\;\,\! 0}      & {\;\,\! 1}             & {\;\,\! 1}\\          
r_7 & {\;\,\! 1}           & {\;\,\! 1}          & {\;\,\! 0}          & {\;\,\! 1}          & {\;\,\! 0}          & {\;\,\! 1}          & {\;\,\! 0}           & {\;\,\! 1}       & {\;\,\! 1}      & {\;\,\! 1}             & {\;\,\! 1}\\
r_8 & {\;\,\! 0}           & {\;\,\! 1}          & {\;\,\! 0}          & {\;\,\! 0}          & {\;\,\! 0}          & {\;\,\! 1}          & {\;\,\! 0}           & {\;\,\! 1}       & {\;\,\! 1}      & {\;\,\! 1}             & {\;\,\! 1}
}
$
\label{fig_g26}
\end{figure*}
\noindent
\begin{center} $M(G_{26})/\{r_5,r_6,r_7,s_4\}$ contains an $M(K_5)$-minor.\end{center}
 {\tt \scriptsize {\bf Command:} ./macek -pGF2 '!contract 5;!contract 6;!contract 7;!contract -4;!minor' g26 '\{grK5,grK33\}'}\\
{\tt \scriptsize {\bf Output: }
TThe \#1 matroid [g26$\sim$c5$\sim$c6$\sim$c7$\sim$c-4] +HAS+ minor \#1 [grK5] in the list \{grK5 grK33\}.}

\subsubsection*{The matroid $M(G_{27})$:}
\FloatBarrier
\begin{figure*}[h]
\raisebox{-10ex}{\resizebox*{0.6\width}{!}{
\psfrag{r1}{\footnotesize $r_1$}
\psfrag{r2}{\footnotesize $r_2$}
\psfrag{r3}{\footnotesize $r_3$}
\psfrag{r4}{\footnotesize $r_4$}
\psfrag{r5}{\footnotesize $r_5$}
\psfrag{r6}{\footnotesize $r_6$}
\psfrag{r7}{\footnotesize $r_7$}
\psfrag{r8}{\footnotesize $r_8$}
\psfrag{s1}{\footnotesize $s_1$}
\psfrag{s2}{\footnotesize $s_2$}
\psfrag{s3}{\footnotesize $s_3$}
\psfrag{s4}{\footnotesize $s_4$}
\psfrag{s5}{\footnotesize $s_5$}
\psfrag{s6}{\footnotesize $s_6$}
\psfrag{s7}{\footnotesize $s_7$}
\psfrag{s8}{\footnotesize $s_8$}
\psfrag{s9}{\footnotesize $s_9$}
\psfrag{s10}{\footnotesize $s_{10}$}
\includegraphics [scale=0.8] {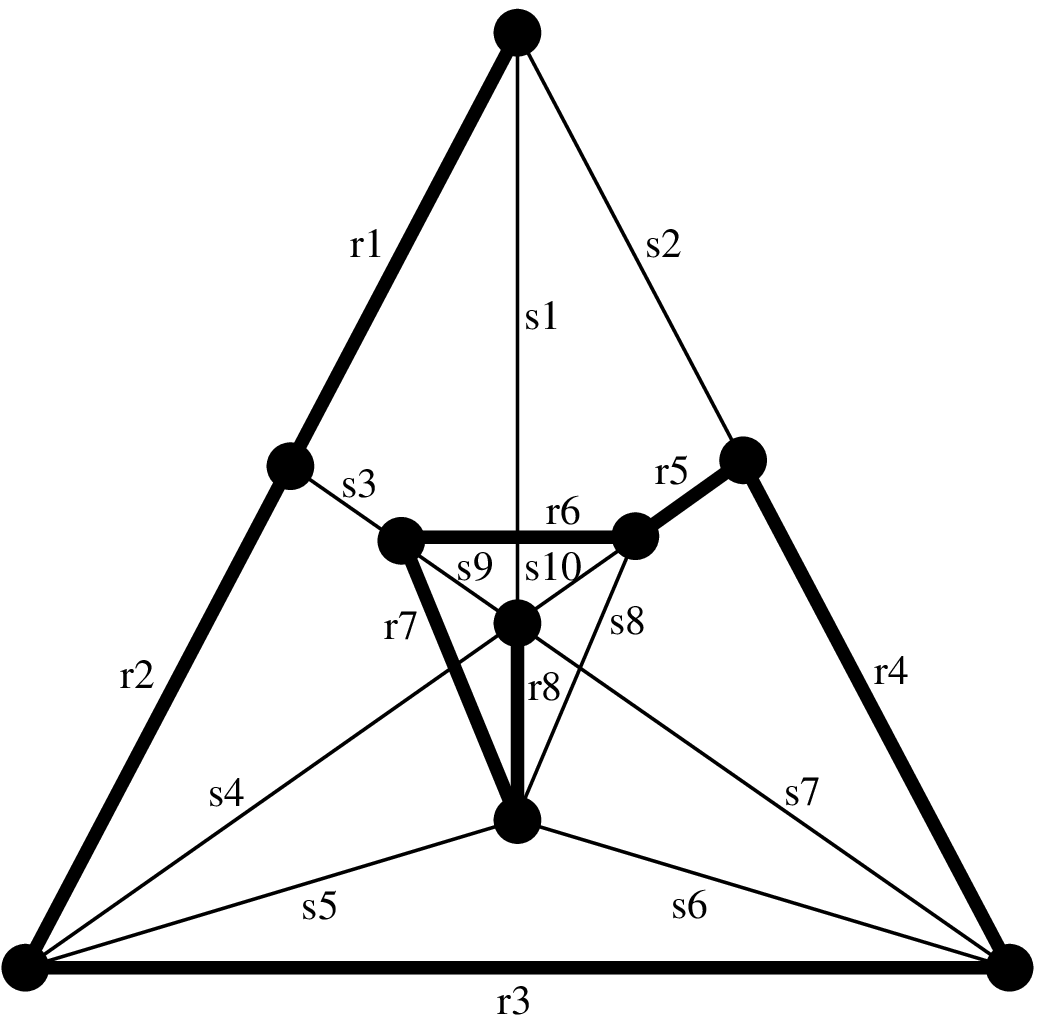}}}
$\mspace{30mu}$
$
g_{27}=
\kbordermatrix{\mbox{}&s_1&s_2&s_3&s_4&s_5&s_6&s_7&s_8&s_9&s_{10}\\
r_1 & {\;\,\! 1}  	   & {\;\,\! 1}          & {\;\,\! 0}          & {\;\,\! 0}          & {\;\,\! 0}          & {\;\,\! 0}          & {\;\,\! 0}           & {\;\,\! 0}       & {\;\,\! 0}      & {\;\,\! 0}   \\
r_2 & {\;\,\! 1}           & {\;\,\! 1}          & {\;\,\! 1}          & {\;\,\! 0}          & {\;\,\! 0}          & {\;\,\! 0}          & {\;\,\! 0}           & {\;\,\! 0}       & {\;\,\! 0}      & {\;\,\! 0}    \\ 
r_3 & {\;\,\! 1}           & {\;\,\! 1}          & {\;\,\! 1}          & {\;\,\! 1}          & {\;\,\! 1}          & {\;\,\! 0}          & {\;\,\! 0}           & {\;\,\! 0}       & {\;\,\! 0}      & {\;\,\! 0}     \\ 
r_4 & {\;\,\! 1}           & {\;\,\! 1}          & {\;\,\! 1}          & {\;\,\! 1}          & {\;\,\! 1}          & {\;\,\! 1}          & {\;\,\! 1}           & {\;\,\! 0}       & {\;\,\! 0}      & {\;\,\! 0}     \\
r_5 & {\;\,\! 1}           & {\;\,\! 0}          & {\;\,\! 1}          & {\;\,\! 1}          & {\;\,\! 1}          & {\;\,\! 1}          & {\;\,\! 1}           & {\;\,\! 0}       & {\;\,\! 0}      & {\;\,\! 0}     \\ 
r_6 & {\;\,\! 1}           & {\;\,\! 0}          & {\;\,\! 1}          & {\;\,\! 1}          & {\;\,\! 1}          & {\;\,\! 1}          & {\;\,\! 1}           & {\;\,\! 1}       & {\;\,\! 0}      & {\;\,\! 1}     \\          
r_7 & {\;\,\! 1}           & {\;\,\! 0}          & {\;\,\! 0}          & {\;\,\! 1}          & {\;\,\! 1}          & {\;\,\! 1}          & {\;\,\! 1}           & {\;\,\! 1}       & {\;\,\! 1}      & {\;\,\! 1}       \\
r_8 & {\;\,\! 1}           & {\;\,\! 0}          & {\;\,\! 0}          & {\;\,\! 1}          & {\;\,\! 0}          & {\;\,\! 0}          & {\;\,\! 1}           & {\;\,\! 0}       & {\;\,\! 1}      & {\;\,\! 1}
}
$
\label{fig_g27}
\end{figure*}
\noindent
\begin{center} $M(G_{27})/\{r_1,r_2,r_3,r_4,s_2\}$ contains an $M(K_5)$-minor.\end{center}
 {\tt \scriptsize {\bf Command:} ./macek -pGF2 '!contract 1;!contract 2;!contract 3;!contract 4;!contract -2;!minor' g27 '\{grK5,grK33\}'}\\
{\tt \scriptsize {\bf Output: }
The \#1 matroid [g27$\sim$c1$\sim$c2$\sim$c3$\sim$c4$\sim$c-2] +HAS+ minor \#1 [grK5] in the list \{grK5 grK33\}.}

\subsubsection*{The matroid $M(G_{28})$:}
\FloatBarrier
\begin{figure*}[h]
\raisebox{-10ex}{\resizebox*{0.6\width}{!}{
\psfrag{r1}{\footnotesize $r_1$}
\psfrag{r2}{\footnotesize $r_2$}
\psfrag{r3}{\footnotesize $r_3$}
\psfrag{r4}{\footnotesize $r_4$}
\psfrag{r5}{\footnotesize $r_5$}
\psfrag{r6}{\footnotesize $r_6$}
\psfrag{r7}{\footnotesize $r_7$}
\psfrag{r8}{\footnotesize $r_8$}
\psfrag{r9}{\footnotesize $r_9$}
\psfrag{s1}{\footnotesize $s_1$}
\psfrag{s2}{\footnotesize $s_2$}
\psfrag{s3}{\footnotesize $s_3$}
\psfrag{s4}{\footnotesize $s_4$}
\psfrag{s5}{\footnotesize $s_5$}
\psfrag{s6}{\footnotesize $s_6$}
\psfrag{s7}{\footnotesize $s_7$}
\psfrag{s8}{\footnotesize $s_8$}
\psfrag{s9}{\footnotesize $s_9$}
\includegraphics [scale=0.8] {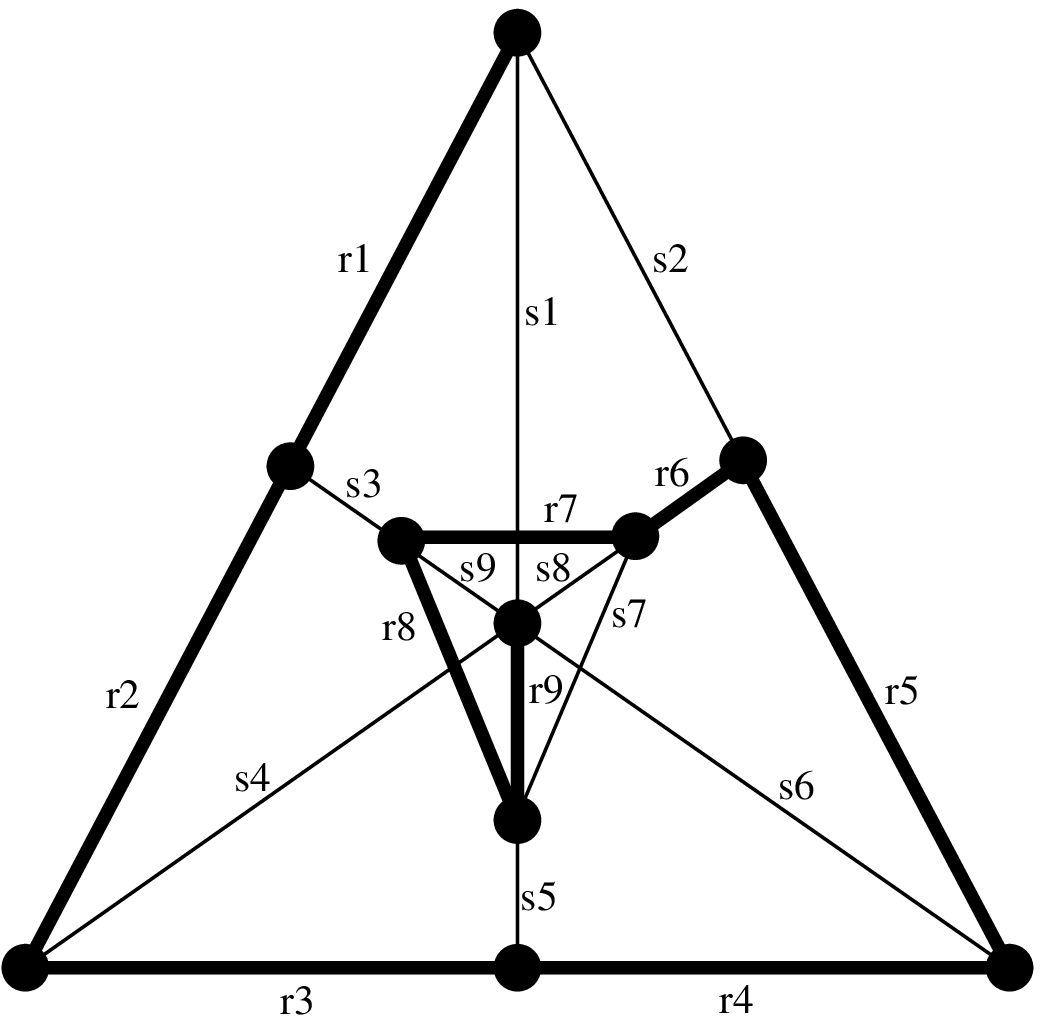}}}
$\mspace{30mu}$
$
g_{28}=
\kbordermatrix{\mbox{}&s_1&s_2&s_3&s_4&s_5&s_6&s_7&s_8&s_9\\
r_1 & {\;\,\! 1}  	   & {\;\,\! 1}          & {\;\,\! 0}          & {\;\,\! 0}          & {\;\,\! 0}          & {\;\,\! 0}          & {\;\,\! 0}            & {\;\,\! 0}        & {\;\,\! 0}\\
r_2 & {\;\,\! 1}           & {\;\,\! 1}          & {\;\,\! 1}          & {\;\,\! 0}          & {\;\,\! 0}          & {\;\,\! 0}          & {\;\,\! 0}            & {\;\,\! 0}        & {\;\,\! 0}\\ 
r_3 & {\;\,\! 1}           & {\;\,\! 1}          & {\;\,\! 1}          & {\;\,\! 1}          & {\;\,\! 0}          & {\;\,\! 0}          & {\;\,\! 0}            & {\;\,\! 0}        & {\;\,\! 0}\\ 
r_4 & {\;\,\! 1}           & {\;\,\! 1}          & {\;\,\! 1}          & {\;\,\! 1}          & {\;\,\! 1}          & {\;\,\! 0}          & {\;\,\! 0}            & {\;\,\! 0}        & {\;\,\! 0}\\
r_5 & {\;\,\! 1}           & {\;\,\! 1}          & {\;\,\! 1}          & {\;\,\! 1}          & {\;\,\! 1}          & {\;\,\! 1}          & {\;\,\! 0}            & {\;\,\! 0}        & {\;\,\! 0}\\ 
r_6 & {\;\,\! 1}           & {\;\,\! 0}          & {\;\,\! 1}          & {\;\,\! 1}          & {\;\,\! 1}          & {\;\,\! 1}          & {\;\,\! 0}            & {\;\,\! 0}        & {\;\,\! 0}\\          
r_7 & {\;\,\! 1}           & {\;\,\! 0}          & {\;\,\! 1}          & {\;\,\! 1}          & {\;\,\! 1}          & {\;\,\! 1}          & {\;\,\! 1}            & {\;\,\! 1}        & {\;\,\! 0}\\
r_8 & {\;\,\! 1}           & {\;\,\! 0}          & {\;\,\! 0}          & {\;\,\! 1}          & {\;\,\! 1}          & {\;\,\! 1}          & {\;\,\! 1}            & {\;\,\! 1}        & {\;\,\! 1}\\
r_9 & {\;\,\! 1}           & {\;\,\! 0}          & {\;\,\! 0}          & {\;\,\! 1}          & {\;\,\! 0}          & {\;\,\! 1}          & {\;\,\! 0}            & {\;\,\! 1}        & {\;\,\! 1}
}
$
\label{fig_g28}
\end{figure*}     
\noindent
\begin{center} $M(G_{28})/\{r_1,r_2,r_3,r_4,r_5, s_2\}$ contains an $M(K_{5})$-minor.\end{center}
{\tt \scriptsize {\bf Command:} ./macek -pGF2 '!contract 1;!contract 2;!contract 3;!contract 4;!contract 5;!contract -2;!minor' g28 '\{grK5,grK33\}'}\\
{\tt \scriptsize {\bf Output: }
 The \#1 matroid [g28$\sim$c1$\sim$c2$\sim$c3$\sim$c4$\sim$c5$\sim$c] +HAS+ minor \#1 [grK5] in the list \{grK5 grK33\}.}

\subsubsection*{The matroid $M(G_{29})$:}
\FloatBarrier
\begin{figure*}[h]
\raisebox{-10ex}{\resizebox*{0.6\width}{!}{
\psfrag{r1}{\footnotesize $r_1$}
\psfrag{r2}{\footnotesize $r_2$}
\psfrag{r3}{\footnotesize $r_3$}
\psfrag{r4}{\footnotesize $r_4$}
\psfrag{r5}{\footnotesize $r_5$}
\psfrag{r6}{\footnotesize $r_6$}
\psfrag{r7}{\footnotesize $r_7$}
\psfrag{r8}{\footnotesize $r_8$}
\psfrag{r9}{\footnotesize $r_9$}
\psfrag{r10}{\footnotesize $r_{10}$}
\psfrag{s1}{\footnotesize $s_1$}
\psfrag{s2}{\footnotesize $s_2$}
\psfrag{s3}{\footnotesize $s_3$}
\psfrag{s4}{\footnotesize $s_4$}
\psfrag{s5}{\footnotesize $s_5$}
\psfrag{s6}{\footnotesize $s_6$}
\psfrag{s7}{\footnotesize $s_7$}
\psfrag{s8}{\footnotesize $s_8$}
\includegraphics [scale=0.8] {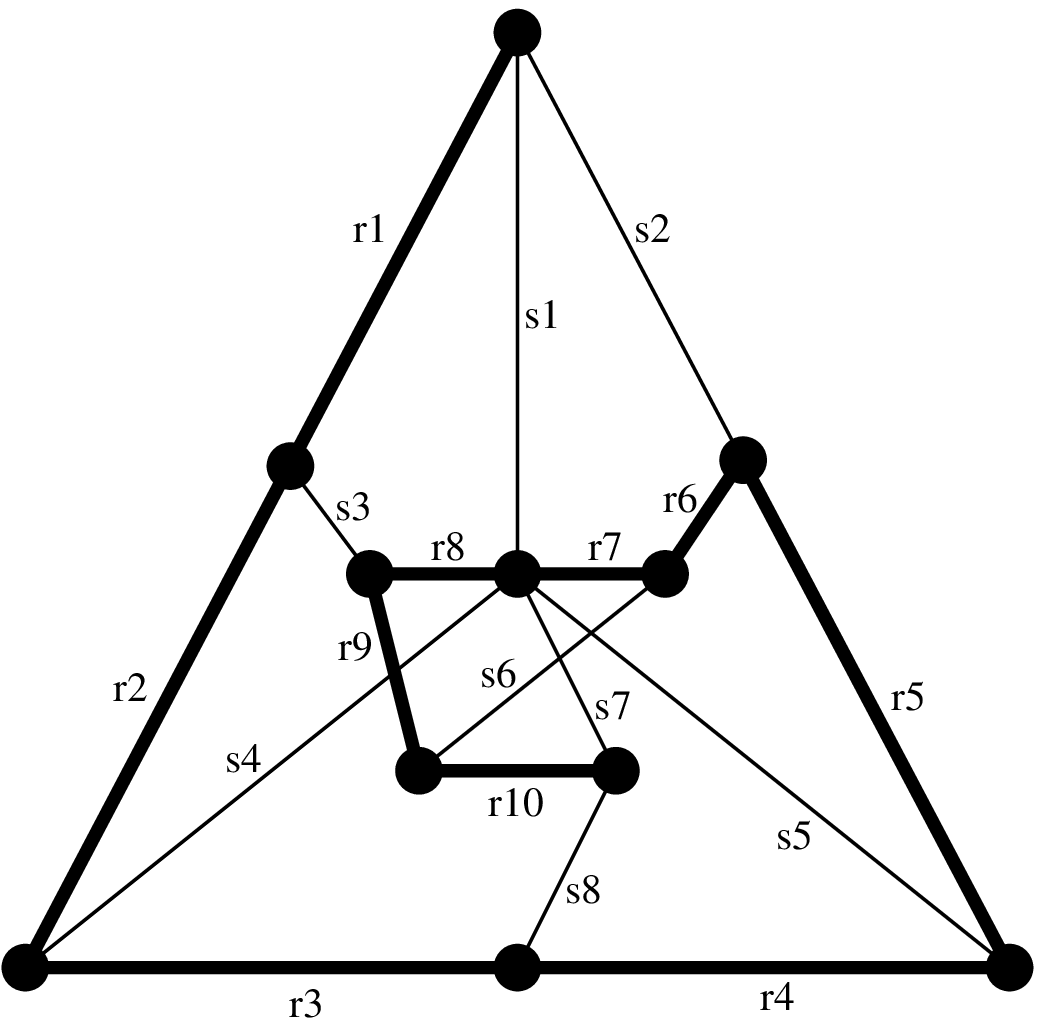}}}
$\mspace{30mu}$
$
g_{29}=
\kbordermatrix{\mbox{}&s_1&s_2&s_3&s_4&s_5&s_6&s_7&s_8\\
r_1    & {\;\,\! 1}           & {\;\,\! 1}          & {\;\,\! 0}          & {\;\,\! 0}          & {\;\,\! 0}          & {\;\,\! 0}          & {\;\,\! 0}         & {\;\,\! 0}          \\
r_2    & {\;\,\! 1}           & {\;\,\! 1}          & {\;\,\! 1}          & {\;\,\! 0}          & {\;\,\! 0}          & {\;\,\! 0}          & {\;\,\! 0}         & {\;\,\! 0}          \\ 
r_3    & {\;\,\! 1}           & {\;\,\! 1}          & {\;\,\! 1}          & {\;\,\! 1}          & {\;\,\! 0}          & {\;\,\! 0}          & {\;\,\! 0}         & {\;\,\! 0}            \\ 
r_4    & {\;\,\! 1}           & {\;\,\! 1}          & {\;\,\! 1}          & {\;\,\! 1}          & {\;\,\! 0}          & {\;\,\! 0}          & {\;\,\! 0}         & {\;\,\! 1}          \\ 
r_5    & {\;\,\! 1}           & {\;\,\! 1}          & {\;\,\! 1}          & {\;\,\! 1}          & {\;\,\! 1}          & {\;\,\! 0}          & {\;\,\! 0}         & {\;\,\! 1}          \\ 
r_6    & {\;\,\! 1}           & {\;\,\! 0}          & {\;\,\! 1}          & {\;\,\! 1}          & {\;\,\! 1}          & {\;\,\! 0}          & {\;\,\! 0}         & {\;\,\! 1}           \\ 
r_7    & {\;\,\! 1}           & {\;\,\! 0}          & {\;\,\! 1}          & {\;\,\! 1}          & {\;\,\! 1}          & {\;\,\! 1}          & {\;\,\! 0}         & {\;\,\! 1}            \\
r_8    & {\;\,\! 0}           & {\;\,\! 0}          & {\;\,\! 1}          & {\;\,\! 0}          & {\;\,\! 0}          & {\;\,\! 1}          & {\;\,\! 1}         & {\;\,\! 1}           \\
r_9    & {\;\,\! 0}           & {\;\,\! 0}          & {\;\,\! 0}          & {\;\,\! 0}          & {\;\,\! 0}          & {\;\,\! 1}          & {\;\,\! 1}         & {\;\,\! 1}      \\
r_{10} & {\;\,\! 0}           & {\;\,\! 0}          & {\;\,\! 0}          & {\;\,\! 0}          & {\;\,\! 0}          & {\;\,\! 0}          & {\;\,\! 1}         & {\;\,\! 1}           
}   
$
\label{fig_g29}
\end{figure*}
\noindent
\begin{center} $M(G_{29})/\{r_1,r_2,r_3,r_4,r_5,s_2\}$ contains an $M(K_{3,3})$-minor.\end{center}
{\tt \scriptsize {\bf Command:}./macek -pGF2 '!contract 1;!contract 2;!contract 3;!contract 4;!contract 5;!contract -2;!minor' g29 '\{grK5,grK33\}'}\\
{\tt \scriptsize {\bf Output: }The \#1 matroid [g5$\sim$c7$\sim$c8$\sim$c9$\sim$c-8] +HAS+ minor \#2 [grK33] in the list \{grK5 grK33\}.}

\subsubsection*{The matroid $R_{15}$:}
\FloatBarrier
\[
r_{15}=
\kbordermatrix{\mbox{}&s_1&s_2&s_3&s_4&s_5&s_6&s_7&s_8\\  
r_1 & {\;\,\! 1}  	   & {\;\,\! 0}          & {\;\,\! 1}          & {\;\,\! 0}          & {\;\,\! 0}          & {\;\,\! 0}         & {\;\,\! 0}        & {\;\,\! 1}      \\
r_2 & {\;\,\! 0}           & {\;\,\! 0}          & {\;\,\! 0}          & {\;\,\! 1}          & {\;\,\! 1}          & {\;\,\! 0}         & {\;\,\! 1}        & {\;\,\! 0}      \\ 
r_3 & {\;\,\! 1}           & {\;\,\! 1}          & {\;\,\! 0}          & {\;\,\! 0}          & {\;\,\! 1}          & {\;\,\! 1}         & {\;\,\! 0}        & {\;\,\! 0}      \\ 
r_4 & {\;\,\! 1}           & {\;\,\! 1}          & {\;\,\! 0}          & {\;\,\! 0}          & {\;\,\! 0}          & {\;\,\! 1}         & {\;\,\! 1}        & {\;\,\! 0}      \\
r_5 & {\;\,\! 0}           & {\;\,\! 1}          & {\;\,\! 1}          & {\;\,\! 1}          & {\;\,\! 1}          & {\;\,\! 0}         & {\;\,\! 0}        & {\;\,\! 0}      \\ 
r_6 & {\;\,\! 0}           & {\;\,\! 1}          & {\;\,\! 1}          & {\;\,\! 1}          & {\;\,\! 1}          & {\;\,\! 1}         & {\;\,\! 0}        & {\;\,\! 0}      \\          
r_7 & {\;\,\! 0}           & {\;\,\! 1}          & {\;\,\! 1}          & {\;\,\! 1}          & {\;\,\! 1}          & {\;\,\! 1}         & {\;\,\! 0}        & {\;\,\! 1}      
}
\]
\noindent
\begin{center} $R_{15}\backslash{\{r_6,r_7,s_8\}}$ contains an $M(K_{3,3})$-minor.\end{center}
{\tt \scriptsize {\bf Command:}./macek -pGF2 '!delete 6;!delete 7;!delete -8;!minor' r15 '\{"grK5;!dual","grK33;!dual"\}'}\\
{\tt \scriptsize {\bf Output: }The \#1 matroid [r15$\sim$d6$\sim$d7$\sim$d-8] +HAS+ minor \#2 [grK33\#] in the list \{grK5\# grK33\#\}.}\\

\subsubsection*{The matroid $R_{16}$:}
\FloatBarrier
\[
r_{16}=
\kbordermatrix{\mbox{}&s_1&s_2&s_3&s_4&s_5&s_6&s_7&s_8\\
r_1 & {\;\,\! 0}  	   & {\;\,\! 1}          & {\;\,\! 1}          & {\;\,\! 0}          & {\;\,\! 1}          & {\;\,\! 0}         & {\;\,\! 0}        & {\;\,\! 0}\\
r_2 & {\;\,\! 0}           & {\;\,\! 0}          & {\;\,\! 0}          & {\;\,\! 0}          & {\;\,\! 1}          & {\;\,\! 1}         & {\;\,\! 1}        & {\;\,\! 0}\\ 
r_3 & {\;\,\! 0}           & {\;\,\! 1}          & {\;\,\! 1}          & {\;\,\! 0}          & {\;\,\! 1}          & {\;\,\! 1}         & {\;\,\! 1}        & {\;\,\! 0}\\ 
r_4 & {\;\,\! 0}           & {\;\,\! 0}          & {\;\,\! 0}          & {\;\,\! 1}          & {\;\,\! 1}          & {\;\,\! 1}         & {\;\,\! 0}        & {\;\,\! 0}\\
r_5 & {\;\,\! 1}           & {\;\,\! 1}          & {\;\,\! 0}          & {\;\,\! 1}          & {\;\,\! 1}          & {\;\,\! 1}         & {\;\,\! 0}        & {\;\,\! 0}\\ 
r_6 & {\;\,\! 1}           & {\;\,\! 0}          & {\;\,\! 1}          & {\;\,\! 1}          & {\;\,\! 0}          & {\;\,\! 0}         & {\;\,\! 0}        & {\;\,\! 0}\\          
r_7 & {\;\,\! 1}           & {\;\,\! 1}          & {\;\,\! 0}          & {\;\,\! 0}          & {\;\,\! 0}          & {\;\,\! 0}         & {\;\,\! 0}        & {\;\,\! 1} \\
r_8 & {\;\,\! 1}           & {\;\,\! 0}          & {\;\,\! 1}          & {\;\,\! 0}          & {\;\,\! 0}          & {\;\,\! 0}         & {\;\,\! 0}        & {\;\,\! 1}
}
\]   
\noindent
\begin{center} $M(R_{16})\backslash{\{r_8,s_1,s_3,s_8\}}$ contains an $M(K_{3,3})$-minor.\end{center}
{\tt \scriptsize {\bf Command:} ./macek -pGF2 '!delete 8;!delete -1;!delete -3;!delete -8;!minor' r16 '\{"grK5;!dual","grK33;!dual"\}'}\\
{\tt \scriptsize {\bf Output: }
The \#1 matroid [r16$\sim$d8$\sim$d-1$\sim$d-3$\sim$d-8] +HAS+ minor \#2 [grK33\#] in the list \{grK5\# grK33\#\}}.

\section*{Acknowledgements}
This research has
been funded by the European Union (European Social Fund - ESF) and Greek national funds through the Operational
Program “Education and Lifelong Learning” of the National Strategic Reference Framework (NSRF) -
Research Funding Program: Thalis. Investing in knowledge society through the European Social Fund.

\bibliographystyle{plain}

\end{document}